\documentclass{amsart}
\usepackage{amsmath,amssymb,amsthm,mathrsfs,amscd}
\usepackage{amssymb,a4}
\usepackage{amsmath,amsthm,amscd}
\usepackage{amsthm}
\usepackage{amsfonts}
\usepackage{color}
\theoremstyle{plain}
\newtheorem{theorem}{Theorem}[section]

\newtheorem{lemma}[theorem]{Lemma}
\newtheorem{remark}[theorem]{Remark}
\newtheorem{proposition}[theorem]{Proposition}
\newtheorem{ex}[theorem]{Example}
\newtheorem{cor}[theorem]{Corollary}

 
\def\g{\mathfrak{g}}

\def\gpoly{\mathfrak{g}[t^{\pm1}]}
\def\glau{L{\mathfrak{g}}}
\def\n{\mathfrak{n}}
\def\h{\mathfrak{h}}

\def\Sl{\mathfrak{sl}}
\def\abel{\mathfrak{a}}
\def\abelp{\mathfrak{a}_+}
\def\abeln{\mathfrak{a}_-}
\def\Lie{\operatorname{Lie}}
\def\lgp{L\mathfrak{g}_+}
\def\lgn{L\mathfrak{g}_-}

\def\lG{LG}
\def\lGp{LG_+}
\def\lGn{LG_-}
\def\A{A}
\def\Ap{A_+}
\def\An{A_-}
\def\hom{LG_+/A_+}
\def\afD{\hat{\Delta}}
\def\afDp{\hat{\Delta}_{>0}}
\def\dcoset{G_+\backslash LG_+/A_+}
\def\liep{\mathcal{L}}
\def\witt{\mathrm{Witt}}
\def\afg{\widetilde{\mathfrak{g}}}
\def\afn{\widetilde{\mathfrak{n}}}
\def\afh{\widetilde{\mathfrak{h}}}
\def\dst{d_{\mathrm{st}}}
\def\dne{d_{\mathrm{ne}}}
\def\dch{d_{\chi}}
\def\dstz{d_{\mathrm{st}(0)}}
\def\dnez{d_{\mathrm{ne}(0)}}
\def\dchz{d_{\chi(0)}}
\def\ad{\operatorname{ad}}
\def\coad{\operatorname{ad^*}}
\def\pd{\partial}
\def\gr{\mathrm{gr}}
\def\deg{\mathrm{deg}}
\def\wt{\mathrm{wt}}
\def\Ker{\operatorname{Ker}}
\def\Im{\operatorname{Im}}
\def\exp{\operatorname{exp}}
\def\Der{\operatorname{Der}}
\def\Hom{\operatorname{Hom}}
\def\Vect{\operatorname{Vect}}
\def\U{\mathcal{U}}
\def\e{\operatorname{e}}
\def\C{\mathbb{C}}
\def\Z{\mathbb{Z}}
\def\W{\mathcal{W}}
\def\lam{\lambda}
\def\vp{\varphi}
\def\ep{\epsilon}
\def\al{\alpha}
\def\bt{\beta}
\def\gam{\gamma}
\def\bal{\bar{\alpha}}
\def\bbt{\bar{\beta}}

\def\dim{\mathrm{dim}\ }
\def\half{{1/2}}
\def\inv{(-|-)}

\def\dr{\mathrm{dR}}
\def\cl{\mathrm{cl}}
\def\L{\mathbb{L}_0^k}
\def\k{\kappa}
\def\tw{\tilde{w}}
\newcommand{\dynkinGtwo}
{\begin{picture}(10,15)
  \put(-10,3){\circle{4}}
  \put(15,3){\circle{4}}
  \put(-10,1.3){\line(1,0){25}}
  \put(-9,3){\line(1,0){22}}
  \put(-10,4.8){\line(1,0){25}}
  \put(7,3){\line(-2,1){10}}
  \put(7,3){\line(-2,-1){10}}
\end{picture}}
\newcommand{\dynkinFfour}
{\begin{picture}(20,15)
  \put(-20,2){\circle{4}}
  \put(0,2){\circle{4}}
  \put(20,2){\circle{4}}
  \put(40,2){\circle{4}}
  \put(-18,2){\line(1,0){16}}
  \put(2,3){\line(1,0){16}}
  \put(2,1){\line(1,0){16}}
  \put(22,2){\line(1,0){16}}
 \put(14,2){\line(-2,1){10}}
 \put(14,2){\line(-2,-1){10}}
\end{picture}}
\newcommand{\dynkinEsix}
{\begin{picture}(40,15)
  \put(-6,0){\circle{4}}
  \put(6,0){\circle{4}}
  \put(18,0){\circle{4}}
  \put(30,0){\circle{4}}
  \put(42,0){\circle{4}}
  \put(18,12){\circle{4}}
  \put(-4,0){\line(1,0){8}}
  \put(8,0){\line(1,0){8}}
  \put(20,0){\line(1,0){8}}
  \put(32,0){\line(1,0){8}}
  \put(18,2){\line(0,1){8}}
\end{picture}}
\newcommand{\dynkinEseven}
{\begin{picture}(70,15)
  \put(6,0){\circle{4}}
  \put(18,0){\circle{4}}
  \put(30,0){\circle{4}}
  \put(42,0){\circle{4}}
  \put(54,0){\circle{4}}
  \put(66,0){\circle{4}}
  \put(30,12){\circle{4}}
  \put(8,0){\line(1,0){8}}
  \put(20,0){\line(1,0){8}}
  \put(32,0){\line(1,0){8}}
  \put(44,0){\line(1,0){8}}
  \put(56,0){\line(1,0){8}}
  \put(30,2){\line(0,1){8}}
\end{picture}}
\newcommand{\dynkinEeight}
{\begin{picture}(70,15)
  \put(0,0){\circle{4}}
  \put(12,0){\circle{4}}
  \put(24,0){\circle{4}}
  \put(36,0){\circle{4}}
  \put(48,0){\circle{4}}
  \put(60,0){\circle{4}}
  \put(72,0){\circle{4}}
  \put(24,12){\circle{4}}
  \put(2,0){\line(1,0){8}}
  \put(14,0){\line(1,0){8}}
  \put(26,0){\line(1,0){8}}
  \put(38,0){\line(1,0){8}}
  \put(50,0){\line(1,0){8}}
  \put(24,2){\line(0,1){8}}
  \put(62,0){\line(1,0){8}}
\end{picture}}
\newcommand{\wtGa}{2~~~~~~0}
\newcommand{\wtGb}{2~~~~~~2}
\newcommand{\wtFa}{0~~~2~~~0~~~0~~~}
\newcommand{\wtFb}{0~~~2~~~0~~~2~~~}
\newcommand{\wtFc}{2~~~2~~~0~~~2~~~}
\newcommand{\wtFd}{2~~~2~~~2~~~2~~~}
\newcommand{\wtEaa}{$2~~~~0~~~~\overset{\text{\normalsize 0}}{2}~~~~0~~~~2$}
\newcommand{\wtEab}{$2~~~~2~~~~\overset{\text{\normalsize 2}}{0}~~~~2~~~~2$}
\newcommand{\wtEac}{$2~~~~2~~~~\overset{\text{\normalsize 2}}{2}~~~~2~~~~2$}
\newcommand{\wtEba}{$2~~~~2~~~~\overset{\text{\normalsize 2}}{0}~~~~2~~~~2~~~~2$}
\newcommand{\wtEbb}{$2~~~~2~~~~\overset{\text{\normalsize 2}}{2}~~~~2~~~~2~~~~2$}
\newcommand{\wtEca}{$0~~~~0~~~~\overset{\text{\normalsize 0}}{2}~~~~0~~~~0~~~~2~~~~0$}
\newcommand{\wtEcb}{$2~~~~0~~~~\overset{\text{\normalsize 0}}{2}~~~~0~~~~2~~~~0~~~~2$}
\newcommand{\wtEcc}{$2~~~~2~~~~\overset{\text{\normalsize 2}}{0}~~~~2~~~~0~~~~2~~~~2$}
\newcommand{\wtEcd}{$2~~~~2~~~~\overset{\text{\normalsize 2}}{2}~~~~2~~~~2~~~~2~~~~2$}

\title[$\W$-algebras and integrable systems]{A geometric construction of integrable Hamiltonian hierarchies associated with the classical affine $\W$-algebras}

\author{Shigenori Nakatsuka}
\address{Graduate School of Mathematical Sciences, The University of Tokyo, 3-8-1 Komaba, Tokyo, Japan 153-8914}
\email{nakatuka@ms.u-tokyo.ac.jp}
\begin{document}

\maketitle

\begin{abstract}
A class of classical affine $\W$-algebras are shown to be isomorphic as differential algebras to the coordinate rings of double coset spaces of certain prounipotent proalgebraic groups. As an application, integrable Hamiltonian hierarchies associated with them are constructed geometrically, generalizing the corresponding result of Feigin-Frenkel and Enriquez-Frenkel for the principal cases.
\end{abstract}

\section{Introduction}
Since the discovery of the Drinfel'd-Sokolov hierarchy \cite{ds84}, a number of integrable systems have been constructed in the same spirits cf. \cite{bdghm93, dghm92, fhm93, fgm96, fgm95}. Those integrable systems are not of finite dimension, which are formulated by Poisson algebras, but of infinite dimension, reflecting classical field theory as their origin.

In \cite{bdk09}, Barakat, De Sole and Kac used \emph{Poisson vertex algebras}, which play a role of Poisson algebras for integrable systems as above, and introduced the notion of \emph{integrable Hamiltonian hierarchies} as a framework of integrability for Poisson vertex algebras. 
Since then De Sole, Kac, and Valeri e.g.\ \cite{dkv13, dkv15a, dkv16a, dkv16b, dkv16c} have studied systematically integrable Hamiltonian hierarchies associated with Poisson vertex algebras, called the {\em classical (affine) $\W$-algebras}, which are obtained as classical limit of vertex algebras, called the (affine) $\W$-algebras \cite{dk06,krw03}, see also \cite{dk13a, suh18}.

The $\W$-algebras are parametrized by a triple $(\g,f,k)$ consisting of a finite dimensional simple Lie algebra $\g$ over $\C$, a nonzero nilpotent element $f\in\g$, and a complex number $k\in\C$, called the \emph{level} and are denoted by $\W^k(\g, f)$. The classical $\W$-algebras are also parametrized by the same data and thus we denote them also by $\W^k(\g, f)$.
De Sole, Kac and Valeri recovered the earlier results on the integrable Hamiltonian hierarchies mentioned above in their study and obtained a culminating result in \cite{dkv16c} on the existence of integrable Hamiltonian hierarchies associated with $\W^k(\g, f)$, which states:
\emph{the classical affine $\W$-algebra $\W^k(\g, f)$ admits an integrable Hamiltonian hierarchy for any finite dimensional simple Lie algebra $\g$ of classical type, nonzero nilpotent element $f$, and $k\neq0$.}

Besides the algebraic theory of integrable Hamiltonian hierarchies mentioned above, there is a result of Feigin-Frenkel \cite{feiginfrenkel95,feiginfrenkel93} and Enriquez-Frenkel \cite{enriqfrenkel97} which constructs the Drinfel'd-Sokolov hierarchies by using a geometric realization of the classical affine $\W$-algebra $\W^1(\g,f_\mathrm{prin})$, where $f_\mathrm{prin}$ denotes the principal nilpotent element in $\mathfrak{g}$. More precisely, they proved that $\W^1(\g,f_{\mathrm{prin}})$ are isomorphic as differential algebras to the coordinate rings of certain double coset spaces of prounipotent proalgebraic groups and that the Drinfel'd-Sokolov hierarchies are induced from natural group actions on these spaces.

The aim of this paper is to generalize the those results to the classical $\W$-algebras $\W^k(\g, f)$ when $f$ satisfies certain properties (see the condition (F) in the below) and the level $k\in\C$ is generic. We note that the choice of $f$ is a special case of the so-called Type I in \cite{bdghm93, dghm92,fhm93}. The integrable Hamiltonian hierarchies obtained in the paper for the case $k=1$ coincide with special cases of those considered in \cite{bf01} by construction, and are special cases of the result \cite{dkv13}. 

Let $\g$ be a finite dimensional Lie algebra over $\C$, and $f\in \g$ be a non-zero nilpotent element. We fix an $\mathfrak{sl}_2$-triple $\{e,h=2x,f\}$ containing $f$. We denote by $\Gamma: \g=\oplus_{j=-d}^d\g_j$ the $\ad_{x}$-grading and by $\g=\h\oplus (\oplus_{\al\in\Delta}\g_\al)$ a root space decomposition of $\g$ which is homogeneous with respect to $\Gamma$. Then the root system $\Delta$ admits an induced grading $\Delta=\sqcup_{j=-d}^d \Delta_j$. Let $\Pi$ denote the subset of $\Delta_{>0}=\sqcup_{j>0} \Delta_j$ consisting of the elements indecomposable in $\Delta_{>0}$. Let $V^k(\g_0)$ denote the universal affine Poisson vertex algebra associated with the reductive Lie algebra $\g_0$ at level $k$ and $F(\g_\half)$ the $\beta\gamma$-system Poisson vertex algebra associated with the symplectic vector space $\g_\half$. Then the $\W$-algebra $\W^k(\g,f)$ is defined as a 0-th cohomology of the BRST complex $C^k(\g,f)$, see Section \ref{affW} for details.

For generic level $k$, the $\W$-algebras are realized as the joint kernel of certain screening operators \cite{gen17}. Our first result (Theorem \ref{main1}) is a Poisson vertex algebra analogue of this result. Namely, for generic $k\in\C$, we realize the classical affine $\W$-algebra $\W^k(\g,f)$ as the Poisson vertex subalgebra of $V^k(\g_0)\otimes F(\g_\half)$ invariant under the  derivations $Q_\al^W$, $(\al\in \Pi)$:
\begin{equation}\label{eq0}\W^k(\g,f) \cong \bigcap_{\al\in \Pi}\Ker\Bigl(Q_\al^W: V^k(\g_0)\otimes F(\g_\half)\rightarrow V^k(\g_0)\otimes F(\g_\half)\Bigr).
\end{equation}
\noindent See \eqref{eq;2} and  \eqref{eq;1} for the definition of $Q_\al^W$.

We now suppose that the pair $(\g,f)$ satisfies the following condition (F):
\begin{enumerate}
\item[(F1)] The grading $\Gamma$ is a $\Z$-grading.
\item[(F2)] There exists an element $y\in\g_d$ such that $s=f +y  t^{-1}\in\gpoly$ is semisimple.
\item[(F3)] The Lie subalgebra $\Ker(\ad_s)\subset\g[t^{\pm1}]$ is abelian and $\Im(\ad_s)\cap \gpoly_0=\g_0$.
\end{enumerate}
Here $\gpoly=\g \otimes \C[t,t^{-1}]$ and $\gpoly=\oplus_{j\in\Z}\gpoly_j$ is the $\Z$-grading given by $\deg(X t^n)= j +(d+1)n$, $(X\in \g_j)$. 
The nilpotent elements $f$ satisfying (F1)-(F2) are called Type I in the literature e.g.\ \cite{fhm93, bdghm93, dghm92}.
We consider the completion $\glau=\g\otimes \C((t))$. It has subalgebras $\lgp$, which is the completion of $\gpoly_{>0}$, and $\lgn=\gpoly_{\leq0}$. We have the corresponding closed subgroups $\lGp$, $\lGn$ of the loop group $\lG$ of $G$.

Let $\abel$ denote the completion of $\Ker(\ad_s)$ in $\glau$, and $\abel_{\pm}=\abel\cap\glau_{\pm}$. Let $A$, (resp.\ $A_{\pm}$) be the closed subgroup of $\lG$ corresponding to $\abel$, (resp.\ $\abel_\pm$). 

The right $A$-action $\lGn\backslash \lG/\Ap\times A\rightarrow \lGn\backslash \lG/\Ap$, induces a Lie algebra homomorphism $\abel\rightarrow \Der(\C[\lGn\backslash\lG/\Ap])$, and so that $\abel\rightarrow \Der(\C[\lGp/\Ap])$ since $\lGp/A_+\subset \lGn\backslash\lG/\Ap$ is an open subset. We denote this action by $a\mapsto a^R$. In particular, we obtain a differential algebra $(\C[\lGp/\Ap],s^R)$. On the other hand, the natural left $\lGp$-action $\lGp\times\lGp/\Ap\rightarrow \lGp/A_+$ induces a Lie algebra homomorphism  $\lgp\rightarrow \Der(\C[\lGp/\Ap])$, which we denote by $X\mapsto X^L$. 

Our second result (Theorem \ref{main2}) is the following geometric interpretation of \eqref{eq0} under the condition (F): there exists an isomorphism of differential algebras 
$$\Psi_k: V^k(\g_0)\rightarrow \C[\lGp/\Ap],$$
such that the derivation $Q_\al^W$ is identified with $X_\al^L$ for some element $X_\al\in \g_\al$, ($\al\in\Pi$).
Since such root vectors $X_\al$ generate a Lie algebra $\g_+=\oplus_{j>0}\g_j$, it implies that $\W^k(\g,f)$ is isomorphic to $\C[\dcoset]$ as a differential algebra (Corollary \ref{doublecoset}). Here $G_+$ is the closed subgroup of $\lGp$ corresponding to $\g_+$.

The action $\abel\rightarrow \Der(\C[\lGp/A_+])$ preserves the subalgebra $\C[\dcoset]$. Thus we obtain a space of mutually commutative derivations of $\W^\k(\g,f)$ as the image of $\abel_-$. We denote this space by $\mathcal{H}^\k(\g,f)$. Our third result (Theorem \ref{main3}) states that $\mathcal{H}^k(\g,f)$ is an integrable Hamiltonian hierarchy associated with $\W^k(\g,f)$.

The double coset spaces $\dcoset$ considered here are embedded into the abelianized Grassmanians $G[t^{-1}]\backslash G((t))/A_+$ as a Zariski open subset. The abelianized Grassmanians have been used to construct Drinfeld-Sokolov hierarchies geometrically \cite{bf01}. As pointed in \emph{loc.cit.}, such a construction implies a strong compatibility of integrable Hamiltonian systems associated with  classical affine $\W$-algebras and Hitchin systems. The author hope to investigate the relationship between classical affine $\W$-algebras and Hitchin systems in future works.

\vspace{3mm}

{\it Acknowledgements}\quad
This paper is the master thesis of the author. He wishes to express his gratitude to Professor Atsushi Matsuo for encouragement throughout this work and numerous advices.
\section{Poisson vertex superalgebras}
\subsection{Poisson vertex superalgebra}
We recall here some basics about Poisson vertex superalgebras and their relation to the theory of integrable systems, following \cite{bdk09,suh18}. We remove the prefix ``super'' whenever we consider the non-super cases. 

A \emph{differential $\C$-superalgebra} is a pair $(V,\pd)$ consisting of a supercommutative $\C$-superalgebra $V$ and an even derivation $\pd$ on it. We denote by $\bar{a}$ the parity of For $a\in V$ and by $\Der(V)$ the set of super derivations of $V$.
A differential $\C$-superalgera of the form
$$V=\C[u_i^{(n)}| i\in I_{\bar{0}}, n\geq 0]\otimes \bigwedge\ _{\oplus_ {i\in I_{\bar{1}}}\C u_i^{(n)}}$$
as $\C$-superalgebras for some index set $I=I_{\bar{0}}\sqcup I_{\bar{1}}$ where $u_i^{(n)}=\pd^n u_i$, is called the \emph{superalgebra of differential polynomials} in the variables $u_i,\ (i\in I)$.

A \emph{Poisson vertex superalgebra} is a triple $(V,\pd,\{-_\lam-\})$ consisting of a differential $\C$-superalgebra $(V,\pd)$ and an even $\C$-bilinear map
\begin{equation}\label{lambda bracket}
\{-_\lam -\}:\ V\times V\rightarrow V[\lam],\quad (f,g)\mapsto \{f_\lam g\}=\sum_{n\geq0} \frac{\lam^n}{n!} f_{(n)}g,
\end{equation}
called the $\lambda$-bracket, satisfying
\begin{eqnarray}
&&\label{sesqui}\{ \pd f_\lam g\}=-\lam \{f_\lam g\},\quad \{f_\lam \pd g\}=(\lam+\pd)\{f_\lam g\},\\
&&\label{skew}\{g_\lam f\}=- (-1)^{\bar{f}\bar{g}}_\leftarrow\{f_{-\pd-\lam} g\},\\  
&&\label{Jacobi}\{f_\lam \{g_\mu h\}\}-(-1)^{\bar{f}\bar{g}}\{g_\mu \{f_\lam h\}\}=\{\{f_\lam g\}_{\lam+\mu} h\},\\
&&\label{rightLeibniz}\{f_\lam gh\}=\{f_\lam g\} h+(-1)^{\bar{g}\bar{h}}\{f_\lam h\} g,\\
&&\label{leftLeibniz}\{fg_\lam h\}=(-1)^{\bar{g}\bar{h}}\{f_{\lam+\pd} h\}_\rightarrow g+(-1)^{\bar{f}(\bar{g}+\bar{h})}\{g_{\lam+\pd}h\}_\rightarrow f,
\end{eqnarray}
for $f,g,h\in V$. Here we denote 
\begin{equation*}
\{f_\lam g\}_\rightarrow= \sum \frac{1}{n!} f_{(n)}g \lam^n,\quad _\leftarrow \{f_\lam g\}=\sum \lam^n\frac{1}{n!} f_{(n)}g.
\end{equation*}
\begin{remark}
\label{1.1}
Given an algebra of differential polynomials $V$ in the variables $u_i$, $(i\in I)$, a linear map $F: \operatorname{span}_\C \{u_i\}_{i\in I}\rightarrow V[\lam]$ uniquely extends to $F:V\rightarrow V[\lam]$ by \eqref{sesqui}, \eqref{leftLeibniz}.
\end{remark}
Let $V$ be a superalgebra of differential polynomials in the variables $u_i$, $(i\in I)$. We denote by 
$$\frac{\pd}{\pd u_i^{(n)}}, \frac{\pd_R}{\pd_R u_i^{(n)}}\in \mathrm{Der}(V)$$
denote the derivation with respect to $u_i^{(n)}$ from the left and the right respectively, (which coincide if $V$ is non-super).
Then the $\lambda$-bracket on $V$ is determined by the values $\{u_{i\lambda}u_j\}$ in the following sense.
\begin{theorem}[\cite{bdk09,suh18}]
\label{1.2}
Let $(V,\pd)$ be a superalgebra of differential polynomials in the variables $u_i$, $(i\in I)$, and $H_{ij}(\lam)$, $(i,j\in I)$, an element of $V[\lam]$ with the same parity as $u_iu_j$.
Then there is a unique Poisson vertex superalgebra structure on $V$ satisfying $\{u_{i\lam}u_j\}=H_{ji}(\lam)$ if and only if the $\lam$-bracket satisfies \eqref{skew} and \eqref{Jacobi} for $u_i,\ (i \in I)$. 
Moreover, the $\lam$-bracket is given by 
$$\{f_\lam g\}= \sum_{\begin{subarray}{c} i,j\in I,\\ n,m\geq 0\end{subarray}} (-1)^{\bar{f}\bar{g}+\bar{i}\bar{j}}\frac{\pd_R g}{\pd_R u_j^{(n)}}(\lam+\pd)^nH_{ji}(\lam+\pd)_\rightarrow(-\lam-\pd)^m\frac{\pd f}{\pd u_i^{(m)}},\quad f,g\in V.$$
\end{theorem}
\begin{ex}[Universal affine Poisson vertex algebra]
\emph{
Given a finite dimensional Lie algebra $L$ over $\C$ and a nondegenerate symmetric invariant bilinear form $\k$ on $L$, let $V^\k(L)$ denote the algebra of differential polynomials in the variables given by a basis of $L$. Then a $\lam$-bracket $\{-_\lam-\}: L\times L\rightarrow V^\k(L)[\lam]$ given by $\{u_\lam v\}=[u,v]+\k(u,v)\lam$, $(u,v\in L)$, defines a Poisson vertex algebra structure on $V^\k(L)$. This is called the universal \emph{affine Poisson vertex algebra} associated with $L$ at level $\k$.}
\end{ex}
Given a Poisson vertex superalgebra $V$, a \emph{Poisson vertex module} over $V$ (cf. \cite {A}) is a vector superspace $M$ which is a $V$-module as a supercommutative algebra $V$ and endowed with an even $\C$-bilinear map $\{-_\lam-\}: V\times M\rightarrow M[\lam]$
satisfying
\begin{equation}
\{\pd f_\lam m\}=-\lam \{f_\lam m\},
\end{equation}
\begin{equation}
\{f_\lam\{g_\mu m\}\}-(-1)^{\bar{f}\bar{g}}\{g_\mu\{f_\lam m\}\}=\{\{f_\lam g \}_{\lam+\mu} m\},
\end{equation}
\begin{equation}
\{f_\lam g\cdot m\}=\{f_\lam g\}\cdot m+(-1)^{\bar{f}\bar{g}}g\cdot \{f_\lam m\},
\end{equation}
\begin{equation}
\{f\cdot g_\lam m\}=(-1)^{\bar{m}\bar{g}}\{f_{\lam+\pd}m\}_\rightarrow g+(-1)^{\bar{f}(\bar{g}+\bar{m})}\{g_{\lam+\pd}m\}_\rightarrow f
\end{equation}
for $f,g\in V$, $m\in M$. Here the right $V$-action $M\times V\rightarrow M$ is defined by the action of $V$ as a supercommutative algebra. 
In this case, we define the $\lam$-bracket 
$$\{-_\lam-\}:M\times V\rightarrow M[\lam],\quad (m,f)\mapsto \{m_\lam f\}=-(-1)^{\bar{m} \bar{f}}_\leftarrow \{f_{-\pd-\lam}m\}.$$
For $m\in M$, the linear map $\{m_\lam-\}:V\rightarrow M[\lam]$ is called an \emph{intertwining operator} and $m$ is called the Hamiltonian.

For a Poisson vertex superalgebra $V$, the vector superspace $\mathrm{Lie}(V)=V/\pd V$ is called \emph{the space of local functionals}. We denote by 
\begin{center}
$\int: V\rightarrow \mathrm{Lie}(V),\quad f\mapsto \int f$
\end{center}
the canonical projection.
\begin{proposition}[\cite{bdk09,suh18}]
\label{1.3}\hspace{1mm}\\
\begin{enumerate}
\item 
The bilinear map
\begin{center}
$\mathrm{Lie}(V)\times \mathrm{Lie}(V)\rightarrow \mathrm{Lie}(V),\quad (\int f,\int g)\mapsto \int \{f_\lam g\}_{\lam=0}$
\end{center}
is well-defined and defines a Lie superalgebra structure. Moreover, if $V$ is even and an algebra of differential polynomials in the variables $\{u_i\}_{i\in I}$, then
\begin{center}
$\left[\int f,\int g\right]=\sum_{i.j \in I}\int\frac{\delta_R g}{\delta_R u_j} \{u_{i\pd} u_j\}_\rightarrow \frac{\delta f}{\delta u_i},$
\end{center}
where $\frac{\delta f}{\delta u_i}=\sum_{n\geq 0}(-\pd)^n \frac{\pd f}{\pd u_i^{(n)}}$ and $\frac{\delta_R f}{\delta_R u_i}=\sum_{n\geq 0}(-\pd)^n \frac{\pd_R f}{\pd_R u_i^{(n)}}$ denote the left and right variational derivative of $f$ with respect to $u_i$.
\item
The Lie superalgebra $\Lie(V)$  acts on $V$ by
\begin{center}  
$\eta:\Lie(V)\rightarrow \Der(V),\quad \int f \mapsto \{f_\lam -\}|_{\lam=0}.$
\end{center}
\end{enumerate}
\end{proposition}
We use the following lemma in the below.
\begin{lemma}[{cf. \cite[Proposition 1.33]{bdk09}}]
\label{1.4}
Let $L$ be a finite dimensional Lie algebra over $\C$ equipped with a nondegenerate symmetric invariant bilinear form $\k$. For the universal affine Poisson vertex algebra $V^\k(L)$, the kernel of $\eta: \Lie(V^\k(L)) \rightarrow \Der(V^\k(L))$ is 
\begin{center}
$\Ker(\eta) =\mathrm{span}_\C \bigl\{\int 1, \int u \mid u\in Z(L)\bigr\},$
\end{center}
where $Z(L)$ is the center of $L$.
\end{lemma}
\begin{proof}
Let $\{u_i\}_{i\in I}$ be a basis of $L$. Suppose $\int F\in \Ker(\eta)\subset \Lie(V^\k(L))$. By Theorem \ref{1.2},
\begin{center}
$\eta\bigl(\int F\bigr)= \sum\limits_{i,j,n} \pd^n \Bigl(([u_i, u_j]+ \k(u_i, u_j)\pd) \frac{\delta F}{\delta u_i}\Bigr) \frac{\pd}{\pd u_j^{(n)}}$.
\end{center}
It follows
$M_j=\sum\limits_{i} ([u_i, u_j]+ \k(u_i, u_j)\pd) \frac{\delta F}{\delta u_i}=0$ for $j\in I$. Define a degree on $V^\k(L)$ by $\deg(u_i^{(n)})=n$ and $\deg(AB)=\deg(A)+\deg(B)$. Then $\delta/\delta u$ preserves the degree. Let $G_i$ be the top degree component of $\delta F/\delta u$. Then the top component $M_j^{\mathrm{top}}$ of $M_j$ is $M_j^{\mathrm{top}}=\sum_i \k(u_i,u_j)\pd G_i$. Since $\k$ is nondegenerate, we obtain $\pd G_i=0$, which implies $G_i\in \C$. Since $V^\k(L)$ is an algebra of differential polynomials, we conclude $F\in\C\oplus L$. (See the proof of \cite[Proposition 1.5]{bdk09}.) 
Set $F=a+ \sum_i b_i u_i$, $(a, b_i\in\C)$. Then $M_j=[\sum_i b_i u_i, u_j]=0$, ($j\in I$), which implies $\sum_i b_i u_i\in Z(L)$.
\end{proof}

Finally, given a Poisson vertex superalgebra $V$, an element $\int f\in\text{Lie}(V)$ is called \emph{integrable} if there exists an infinite dimensional abelian Lie subsuperalgebra $\mathcal{H}$ of $\text{Lie}(V)$ which contains $\int f$. In this case, $\mathcal{H}$ is called an \emph{integrable Hamiltonian hierarchy} associated with $V$.
\subsection{Differential graded Poisson vertex superalgebra}
A \emph{differential graded Poisson vertex superalgebra} (d.g.\ Poisson vertex superalgebra) is a pair $(V,d)$ consisting of a Poisson vertex superalgebra $V=\oplus_{n\in \Z}V_n$ and a linear map $d:V\rightarrow V$, called the differential, satisfying
\begin{itemize}
\item $V$ is a $\Z$-graded Poisson vertex superalgebra, i.e., $V=\oplus_{n\in\Z}V_n$ as a vector superspace satisfying 
$$V_n\cdot V_m\subset V_{n+m},\quad \{V_{n\lam}V_m\}\subset V_{n+m}[\lam],$$
\item the linear map $d$ is of homogeneous parity and satisfies $d^2=0$, 
$$d:V_n\rightarrow V_{n+1},$$
$$d(ab)=d(a)\cdot b+(-1)^{\bar{d}\bar{a}}a\cdot d(b),\quad d(\{a_\lam b\})=\{d(a)_\lam b\}+(-1)^{\bar{d}\bar{a}}\{a_\lam d(b)\}.$$
\end{itemize} 
The cohomology $H^*(V;d)=\oplus_{n\in \Z}H^n(V;d)$ inherits a $\Z$-graded Poisson vertex algebra structure. Moreover, $H^0(V;d)$ is a Poisson vertex subsuperalgebra and $H^n(V;d)$, $(n\in\Z)$, is a Poisson vertex module over $H^0(V;d)$. In the sequel, we also use the notion of a \emph{differential graded vertex superalgebra}. The definition is similar and therefore we omit the details.
\subsection{Classical limit}\label{classical limit}
Let $V$ be a vertex superalgebra over a polynomial ring $\C[\epsilon]$. Suppose that $V$ is free as a $\C[\epsilon]$-module and the $\lambda$-bracket satisfies 
\begin{equation}\label{condition}
[V_\lam V]\subset \epsilon V.
\end{equation}
Define the vector superspace $V^{\mathrm{cl}}=V/\epsilon V$ and let $V\rightarrow V^{\mathrm{cl}}$, ($f\mapsto \bar{f}$) denote the canonical projection. Then 
$$V^{\mathrm{cl}}\times V^{\mathrm{cl}}\rightarrow V^{\mathrm{cl}},\quad (\bar{f},\bar{g})\mapsto \overline{f_{(-1)}g}$$
is well-defined and defines an associative supercommutative algebra structure on $V^{\mathrm{cl}}$. Since the translation operator $\pd$ of $V$ preserves $\epsilon V$, it induces a linear map
$$\pd: V^{\mathrm{cl}}\rightarrow V^{\mathrm{cl}},\quad \pd\bar{f}\mapsto \overline{\pd(f)},$$
which is a derivation of $V^{\mathrm{cl}}$.
Since the $\epsilon V$ is an ideal of $V$ by \eqref{condition}, the $\lam$-bracket of $V$ induces a bilinear map
$$\{-_\lam -\}: V^{\mathrm{cl}}\times V^{\mathrm{cl}}\rightarrow V^{\mathrm{cl}}[\lam],\quad (\bar{f} \bar{g})\mapsto\overline{[f_\lam g]}.$$
The triple $(V^{\mathrm{cl}},\pd, \{-_\lam -\})$ defines a Poisson vertex superalgebra, called the \emph{classical limit} of $V$ (cf. \cite{dk06,fbz}).
\section{Screening operators for classical affine $\W$-algebras}
In this section, we describe the classical affine $\W$-algebras by using screening operators. They will be obtained as a classical limit of the screening operators for the affine $\W$-algebras obtained in \cite{gen17}. We will use the same notation for Poisson vertex algebras as vertex algebras since there will be no confusion.
\subsection{Affine $\W$-algebras}\label{affW}
Let $\g$ be a finite dimensional simple Lie algebra over $\C$ with the normalized symmetric invariant bilinear form $\k=\inv$. Let $f\in\g$ be a nonzero nilpotent element, fix an $\Sl_2$-triple $\{e,h=2 x,f\}$ containing $f$ and denote by $\Gamma: \g=\oplus_{-d\leq j\leq d} \g_j$ the $\frac{1}{2}\Z$-grading given by  $\ad_x$, with $d$ the largest number such that $\g_d\neq0$. We fix a triangular decomposition   
$\g=\n_+\oplus\h\oplus \n_-$ so that $x\in \h$, $\g_{>0}=\oplus_{j>0}\g_j\subset \n_+$, and $\g_{<0}=\oplus_{j<0}\g_j\subset \n_-$. Let $\g=\h\bigoplus\oplus_{\al\in\Delta}\g_\al$ be a root space decomposition, $\Delta_j=\{\al\in\Delta\mid \g_\al\subset \g_j\}$, and $\Delta_{>0}=\sqcup_{j>0} \Delta_j$. Fix a nonzero root vector $e_\al$ in $\g_\al$ and a basis $e_i$, $(i\in I)$, of $\h$. Then $e_\al$, $(\al\in I\sqcup \Delta)$, form a basis of $\g$. We denote by $\{e_{\bar{\al}}\}_{\al\in I\sqcup \Delta}$ its dual basis of $\g$ with respect to $\k$. Let $c_{\al,\bt}^\gam$ denote the structure constants of $\g$, i.e., $[e_\al,e_\bt]=\sum_{\gam\in I \sqcup\Delta} c_{\al,\bt}^\gam e_\gam$.  

Let  $V^k(\g)$ be the universal affine vertex algebra of $\g$ at level $k$, generated by the even elements $e_\al$, $(\al\in I\sqcup \Delta)$, with $\lam$-bracket 
$[e_{\al\lam}e_\bt]=[e_\al,e_\bt]+k(e_\al|e_\bt)\lam$. 
Let $F^{\operatorname{ch}}(\g_{>0})$ be the charged free fermion vertex superalgebra associated with the symplectic odd vector superspace $\g_{>0}\oplus \g_{>0}^*$, generated by the odd elements $\vp_\al$, $\vp^\al$, $(\al\in\Delta_{>0})$, with $\lam$-bracket
$[\vp_{\al\lam}\vp^\bt]=\delta_{\al,\bt}$, $[\vp_{\al\lam}\vp_\bt]=[\vp^\al_\lam\vp^\bt]=0$.
Let $F(\g_\half)$ be the $\beta\gamma$-system vertex algebra associated with the symplectic vector space $\g_\half$, generated by $\Phi_\al$, ($\al\in \Delta_\half$), with $\lam$-bracket $[\Phi_{\al\lam}\Phi_\beta]=\chi([e_\al,\beta])$, where $\chi(-)=(f\mid-)$.

The affine $\W$-algebra $\W^k(\g,f)$ associated with the triple $(\g,f, k)$, ($k\in\C$), is the vertex algebra defined as the $0$-th cohomology of the differential graded vertex algebra
$$C^k(\g,f)=V^k(\g)\otimes F^{\operatorname{ch}}(\g_{>0})\otimes F(\g_\half),$$
with differential 
$$d_{(0)}=\Bigl[\sum_{\al\in \Delta_{>0}}\big((e_\al+\Phi_\al+\chi(e_\al) \big)\vp^\al-\frac{1}{2} \sum_{\al,\bt,\gam\in\Delta_{>0}} c_{\al,\bt}^\gam \vp_\gam \vp^\al\vp^\bt_\lam-\Bigr]|_{\lam=0},$$
called the \emph{BRST complex}.
The grading $C^k(\g,f)=\oplus_{n\in\Z}C_n^k(\g,f)$ is given by $\gr(e_\al)=\gr(\Phi_\bt)=0$, $\gr(\vp^\al)=-\gr(\vp_\al)=1$ with $\gr(AB)=\gr(A)+\gr(B)$ and $\gr(\pd A)=\gr(A)$. Then we have $d_{(0)}: C_n^k(\g,f)\rightarrow C_{n+1}^k(\g,f)$. We have $H^n(C^k(\g,f))=0$ for $n\neq 0$ (see \cite{kw04,kw05}). 
\subsection{Classical limit}\label{classical limit}

By \cite{kw04,kw05}, we have vertex subsuperalgebras $C^k_\pm(\g,f)$, which gives a decomposition of a complex $C^k(\g,f)=C^k_-(\g,f)\otimes C^k_+(\g,f)$ and satisfies  $H^n(C^k_+(\g,f))\cong \delta_{n,0} \C$. Thus $H^0(C^k_-(\g,f))\cong \W^k(\g,f)$. As a vertex superalgebra, $C^k_-(\g,f)$ is generated by
\begin{eqnarray*}
&&J^u=u+\sum_{\bt,\gam\in \Delta_{>0}} c_{u,\bt}^\gam \vp_\gam \vp^\bt,\ (u\in\g_{\leq0}),\\
&&\Phi_\al, (\al\in \Delta_\half),\quad \vp^\al, (\al\in \Delta_{>0}).
\end{eqnarray*}

Following \cite{gen17}, we introduce the \emph{classical affine $\W$-algebra} as the cohomology of the differential graded Poisson vertex algebra in the classical limit of $C^k_-(\g,f)$. 

Suppose $k+h^\vee\neq0$. Set $\ep=\frac{k'}{k+h^\vee}$, $(k'\in\C\backslash\{0\})$, 
$\bar{J}^u=\ep J^u$, $(u\in \g_{\leq0})$, and $\bar{\Phi}_\al=\ep \Phi_\al$, $(\al\in\Delta_\half$). Then we have 
\begin{equation}\label{ope1}
[\bar{J}^u_\lam \bar{J}^v]=\ep\bigl(\bar{J}^{[u,v]}+\big(k'(u|v)+o(\ep))\lam\bigr),\quad
[\vp^\al_\lam \bar{J}^u]=\ep\sum_{\bt\in \Delta_{>0}}c_{u,\bt}^\al \vp^\bt,
\end{equation}
\begin{equation}\label{ope2}
[\bar{\Phi}_{\al\lam}\bar{\Phi}_\bt]=\ep\cdot\chi([e_\al,e_\bt]),\quad 
[\bar{J}^u_\lam\bar{\Phi}_\al]=[\vp^\al_\lam\vp^\bt]=[\vp^\al_\lam\bar{\Phi}_\bt]=0,
\end{equation}
(cf. \cite{gen17}). Viewing $\epsilon$ as an indeterminate in \eqref{ope1}, \eqref{ope2}, we obtain a vertex superalgebra $\tilde{C}^{k}_-$ over the polynomial ring $\C[\epsilon]$. By Section \ref{classical limit}, we obtain a Poisson vertex superalgebra $\tilde{C}^{k}_-/\epsilon \tilde{C}^{k}_-$, which we denote by $C_{k'}^\mathrm{cl}=C_{k'}^\mathrm{cl}(\g,f)$. We have an isomorphism
$$C_{k'}^\mathrm{cl}\cong V^{k'}(\g_{\leq0})\otimes F(\g_\frac{1}{2})\otimes \mathrm{Sym}(\C[\pd]\g_{>0}^*)$$
of Poisson vertex superalgebra where
$V^{k'}(\g_{\leq0})$ is the universal affine Poisson vertex algebra generated by $\g_{\leq0}$ with $\lam$-bracket $\{u_{\lam} v \}=[u,v]+k'(u|v)\lam$, $F(\g_\half)$ the $\beta\gamma$-system Poisson vertex algebra generated by $\Phi_\al$, $(\al\in\Delta_\half)$, with $\lam$-bracket
$\{\Phi_{\al\lam} \Phi_\bt \}=( f|[e_\al,e_\bt])$, and $\mathrm{Sym}(\C[\pd]\g_{>0}^*)$ the differential $\C$-superalgebra generated by odd elements $\vp^\al$, $(\al\in \Delta_{>0})$, which satisfy 
$$\{\vp^\al_{\lam} u\}=\sum_{\bt\in\Delta_{>0}} c^\al_{u,\bt} \vp^\bt,\quad \{\vp^\al_\lam \vp^\bt\}=\{ \vp^\bt_\lam\Phi_\al \}=0.$$

Decompose the differential $d_{(0)}$ as $d_{(0)}=\dstz+\dnez+\dchz$ where
$$\dst=\sum_{\al\in\Delta_{>0}}e_\al \vp^\al-\frac{1}{2}\sum_{\begin{subarray}{c}\al,\bt\in \Delta_{>0}\\ \gam\in \Delta_{>0}\end{subarray}}c_{\al,\bt}^\gam \vp_\gam \vp^\al \vp^\bt,\ \dne=\sum_{\al\in \Delta_\frac{1}{2}}\Phi_\al \vp^\al,\ \dch=\sum_{\al\in\Delta_{>0}} \chi(e_\al)\vp^\al.$$
Then we have 
$$[d_{\mathrm{st}\lam}\bar{J}^u]=-\sum_{\begin{subarray}{c}\al\in I\sqcup \Delta_{\leq 0}\\ \bt\in\Delta_{>0}\end{subarray}} c_{u,\bt}^\al \bar{J}^{e_\al}\vp^\bt+\sum_{\bt\in \Delta_{>0}}(k'(u|v)(\pd+\lam)+o(\ep))\vp^\bt,$$
$$[d_{\mathrm{ne}\lam}\bar{J}^u]=\sum_{\begin{subarray}{c}\al\in\Delta_\half\\ \bt\in\Delta_{>0}\end{subarray}} c_{\bt}^\al \bar{\Phi}_\al\vp^\bt,\quad [\frac{1}{\epsilon}d_{\chi\lam}\bar{J}^u]=\sum_{\bt\in\Delta_{>0}}\chi([u,e_\bt])\vp^\bt,$$
$$[d_{\mathrm{st}\lam}\vp^\al]=-\frac{1}{2} \sum_{\bt,\gam\in \Delta_{>0}}c_{\bt,\gam}^\al \vp^\bt \vp^\gam,\quad [d_{\mathrm{ne}\lam} \bar{\Phi}_\al]=\sum_{\bt\in \Delta_\half}\chi([e_\bt,e_\al])\vp^\bt,$$
$$[d_{\mathrm{ne}\lam}\vp^\al]=[d_{\chi\lam}\vp^\al]=[d_{\mathrm{st}\lam}\bar{\Phi}_\al]=[\frac{1}{\epsilon}d_{\chi\lam}\bar{\Phi}_\al]=0.$$
The differential $d_{(0)}=\{(\dst+\dne+\dch)_\lam-\}|_{\lam=0}$ is given by 
$$\{d_{\mathrm{st}\lam}-\}, \{d_{\mathrm{ne}\lam}-\}, \{d_{\chi\lam}-\}:
C_k^\cl(\g,f) \rightarrow C_k^\cl(\g,f)[\lam],$$
which satisfy
$$\{d_{\mathrm{st}\lam} u \}=-\sum_{\begin{subarray}{c}\al\in I \sqcup\Delta_{\leq0}\\ \bt\in\Delta_{>0} \end{subarray}} c^\al_{u,\bt} e_\al\vp^\bt +\sum_{\bt\in\Delta_{>0}}k(u|e_\bt)(\pd+\lam)\vp^\bt, $$
$$\{d_{\mathrm{st}\lam} e_\al \}=0,\quad \{d_{\mathrm{st}\lam} \vp^\al \}=-\frac{1}{2} \sum_{\bt,\gam\in\Delta_{>0}}c^\al_{\bt,\gam} \vp^\bt \vp^\gam, $$
$$ \{d_{\mathrm{ne}\lam} u \}=\sum_{\begin{subarray}{c}\al\in\Delta_\half\\ \bt\in\Delta_{>0}\end{subarray}}c^\al_{u,\bt} e_\al\vp^\bt,\quad
\{d_{\mathrm{ne}\lam}e_\al \}=\sum_{\bt\in\Delta_\half}(f|[e_\bt,e_\al])\vp^\bt,\quad \{d_{\mathrm{ne}\lam}\vp^\al \}=0,$$
and
$$\{d_{\chi\lam}u \}=\sum_{\bt\in\Delta_{>0}}(f|[u,e_\bt])\vp^\bt,\quad \{d_{\chi\lam}\vp^\al \}=\{d_{\chi\lam}e_\al \}=0,$$
and are extended to $C_k^\mathrm{cl}$ by \eqref{sesqui}, \eqref{rightLeibniz} (see Remark \ref{1.1}). 
The 0-th cohomology of $C_k^\cl(\g,f)$, which we denote by $\W^k(\g,f)=H^0( C_k^\cl(\g,f))$, is a Poisson vertex algebra called \emph{the classical affine $\W$-algebra associated with $(\g,f,k)$} (\cite{gen17}). Note that $H^n( C_k^\cl(\g,f))=0$ holds for all $n\neq0$.
\subsection{Screening operators}

Introduce another grading wt on $C_k^\cl$ by 
$$\wt(u)= -2j, (u\in\g_j),\quad \wt(\Phi_\al)=0,\ (\al\in\Delta_\half),\quad \wt(\vp^\al)=2j \ (\al\in\Delta_j),$$
$$\wt(\pd A)=\wt(A),\ \text{and}\ \wt(AB)=\wt(A)+\wt(B). $$
and  a decreasing filtration $\{F_pC_k^\cl \}_{p\geq0}$ on $C_k^\cl$ by 
$$F_pC_k^\cl=\text{span}\{A\in C_k^\cl|\wt(A)\geq p \},$$
This filtration is exhaustive, separated, and compatible with the grading of $C_k^\cl$ as a complex. The associated spectral sequence $\{E_r, d_r\}_{r\geq0}$ has the differentials
$$d_0=d_{\mathrm{st}(0)}, \quad d_1=d_{\mathrm{ne}(0)}, \quad d_2=d_{\chi(0)},\quad d_r=0,\ (r\geq3),$$
and thus converges at $r=3$.
We will describe $\W^k(\g,f)=H^0(C_k^\cl(\g,f))$ by using it. Since the calculation is straightforward, we omit the details. (The analogous argument for vertex algebras can be found in \cite{gen17}.)

To calculate $E_1=H^*(C_k^\cl;d_0)$, notice that $d_0$ acts by $0$ on $V^k(\g_0)\otimes F(\g_\half)$ and that $(\bigwedge(\oplus_{\al\in\Delta_{>0}}\C \vp^\al), d_0)$ is a subcomplex isomorphic to the Chevalley-Eilenberg complex of the Lie algebra $\g_{>0}$ with coefficients in the trivial representation $\C$. Thus $(V^k(\g_0)\otimes F(\g_\half)\otimes\bigwedge(\oplus_{\al\in\Delta_{>0}}\C \vp^\al), d_0)$ is a subcomplex, whose cohomology is $V^k(\g_0)\otimes F(\g_\half)\otimes H^*(\g_{>0};\C)$. 
\begin{lemma}{\quad}
\begin{enumerate}
\item 
The natural map $V^k(\g_0)\otimes F(\g_\half)\otimes H^*(\g_{>0};\C)\rightarrow E_1$ is an isomorphism of graded vector spaces for generic $k\in\C$.
\item
The isomorphism $E_1^{(0)}\cong V^k(\g_0)\otimes F(\g_\half)$ is an isomorphism of Poisson vertex algebras.
\item
Each cohomology $E_1^{(n)}$ is a Poisson vertex module over $E_1^{(0)}$. Moreover, $E_1^{(n)}$ is isomorphic to $V^k(\g_0)\otimes F(\g_\half)\otimes H^n(\g_{>0};\C)$ as vector spaces.
\end{enumerate}
\end{lemma}

Let us describe the Poisson vertex modules $E_1^{(n)}$ more explicitly. Recall that the coadjoint representation on $\g_{>0}^*$ of $\g_0$ induces a representation of $\g_0$ on  $H^n(\g_{>0},\C)$ as described as follows. (cf. \cite[Chapter 3]{kumar02}) Let $W$ denote the Weyl group of $\g$ and set $I_0=\{i\in I| \al_i\in \Delta_0\}$, $W'_0=\{w\in W| w \Delta_0^+\subset \Delta^+\}$. Let $*: W\times \h^*\rightarrow\h^*$ denote the shifted action of $W$ and $l:W\rightarrow \Z_{\geq0}$ the length function. Then there is an isomorphism of $\g_0$-modules
$$H^n(\g_{>0};\C)\cong \bigoplus_{\begin{subarray}{c}w\in W'_0\\ l(w)=n\end{subarray}}L_0(w^{-1}*0),$$
where $L_0(w^{-1}*0)$ is the integrable highest weight $\g_0$-module with highest weight $w^{-1}*0$.

For a $\g_0$-module $M$, set $\mathbb{M}^k=V^k(\g_0)\otimes_\C M$. The space $\mathbb{M}^k$ has a unique Poisson vertex module over $\mathbb{M}^k$ such that $V^k(\g_0)$ acts as a commutative algebra by multiplication on the first component and the $\lam$-bracket 
$\{-_\lam-\}:V^k(\g_0)\otimes \mathbb{M}^k\rightarrow \mathbb{M}^k[\lam]$ satisfies $\{a_\lam b\otimes m\}=\{a_\lam b\}\otimes m+b \otimes a\cdot m$ for $a\in \g_0$, $b\in V^k(\g_0)$,  and $m\in M$. We denote by $\L(w^{-1}*0)$ the Poisson vertex module obtained from $L_0(w^{-1}*0)$. 
\begin{lemma}
There is an isomorphism 
$$E_1^{(n)}\cong  \bigoplus_{\begin{subarray}{c}w\in W'_0\\ l(w)=n\end{subarray}}\L(w^{-1}*0)\otimes F(\g_\half),$$
as Poisson vertex modules over  $V^k(\g_0)\otimes F(\g_\half)$.
\end{lemma}
In particular, we have $E_1^{(1)}\cong  \bigoplus_{i\in I\backslash I_0}\L(-\al_i)\otimes F(\g_\half)$ and the subspace $L_0(-\al_i)$ is identified as
$$L_0(-\al_i)\cong \bigoplus_{\bt\in [\al_i]}\C \vp^\bt\subset E_1^{(1)},$$
where $[\al_i]=\Delta_{>0}\cap (\al_i+Q_0)$ and $Q_0$ denotes the root lattice of $\Delta_0$. For $\al\in [\al_i]$ with $i\in I\backslash I_0$, we have
\begin{equation}\label{eq;3}
\pd \vp^\al=\frac{1}{k}\sum_{\begin{subarray}{c}\bt\in [\al],\\ \gam\in I\sqcup \Delta_0 \end{subarray}}c_{\bt,\gam}^\al e_{\bar{\gam}} \vp^\bt.
\end{equation}
Let us describe the differentials on $E_1$ induced from $d_{\mathrm{ne}(0)}$ and $d_{\chi(0)}$.
Consider the intertwining operators $Q_i^W: V^k(\g_0)\otimes F(\g_\half)\rightarrow \L(-\al_i)\otimes F(\g_\half)$ given by 
\begin{eqnarray*}
Q_i^W=\begin{cases}{}
\sum_{\bt\in[\al_i]} \{\Phi_\bt \vp^\bt_\lam- \}|_{\lam=0},&(i\in I_\half),\\
\sum_{\bt\in[\al_i]} \{(f|e_\bt) \vp^\bt_\lam- \}|_{\lam=0},& (i \in I_1).
\end{cases}
\end{eqnarray*}
Then we have:
\begin{lemma}
The differentials on $E_1$ induced by $d_{\mathrm{ne}(0)}$ and $d_{\chi(0)}$ are given by 
$$d_{\mathrm{ne}(0)}=\sum_{i\in I_\half} Q_i^W\ \text{and}\ d_{\chi(0)}=\sum_{i\in I_1} Q_i^W.$$
\end{lemma}
Recall that the complex $C_k^\mathrm{cl}=C_k^\mathrm{cl}(\g,f)$ is $\Z_{\geq0}$-graded and that $H^n(C_k^\mathrm{cl}(\g,f)\cong \delta_{n,0}\W^k(\g,f)$. Then we see that $\W^k(\g,f)$ is a subalgebra of the 0-th degree $C_{k,0}^\mathrm{cl}=V^k(\g_{\leq0})\otimes F(\g_\half)$. Since $\mathcal{I}=V^k(\g_{<0})\otimes F(\g_\half)$ is a Poisson vertex ideal of $C_{k,0}^\mathrm{cl}$, we obtain a homomorphism of Poisson vertex algebras
$$\W^k(\g,f)\rightarrow C_{k,0}^\mathrm{cl}/\mathcal{I}\cong V^k(\g_0)\otimes F(\g_{\half}).$$
It is injective and, by using the differentials $d_{\mathrm{ne}(0)}$ and $d_{\chi(0)}$, the image is described as in the following theorem.
\begin{theorem}
\label{2.2}
For generic $k\in\C$, there is an isomorphism 
$$j: \W^k(\g,f)\cong \bigcap\limits_{i\in I_\half\sqcup I_1} \Ker\Bigl(Q_i^W:V^k(\g_0)\otimes F(\g_\half)\rightarrow \L(-\al_i)\otimes F(\g_\half)\Bigr),$$
of Poisson vertex algebras.
\end{theorem}
\noindent The operators  $\{ Q_i^W\}_ {i\in I_\half\sqcup I_1}$ are called the \emph{screening operators} for $\W^k(\g,f).$
We note that the level $k=1$ is generic \cite{gen17}.
The inclusion $j:\ \W^k(\g,f) \rightarrow V^k(\g_0)\otimes F(\g_\half)$ in Theorem \ref{2.2} induces a Lie algebra homomorphism between their spaces of local functionals
$$j_*: \Lie\bigl(\W^k(\g,f)\bigr) \rightarrow \Lie\bigl(V^k(\g_0)\otimes F(\g_\half)\bigr).$$
(See Proposition \ref{1.3}.)
\begin{lemma}
\label{2.3}
The Lie algebra homomorphism $j_*$ is injective.
\end{lemma} 
\begin{proof}
It is easy to see
$$\Ker\Bigl(\pd:\L(-\al_i)\otimes F(\g_\half) \rightarrow \L(-\al_i)\otimes F(\g_\half)\Bigr) =0,\quad (i\in I_\half\sqcup I_1).$$
Take an element $f\in \W^k(\g,f)$ such that $\int f\in\Ker j_*$. Then there exists an element $G\in V^k(\g_0)$ such that $j(f)=\pd g$. Let $H_i^W$ denote the Hamiltonian of $Q_i^W$. Then we have
$$0=Q_i^W j(f)=\{H^W_{i\lam}j(f)\}|_{\lam=0}=\{H^W_{i\lam}\pd g\}|_{\lam=0}=\pd \{H^W_{i\lam}g\}|_{\lam=0},$$ 
and so that  $\{H^W_{i\lam}g\}|_{\lam=0}\in\Ker(\pd: \L(\al_i)\otimes F(\g_\half)) \ \rightarrow \L(\al_i)\otimes F(\g_\half))=0$. Therefore, we obtain $g\in j(\W^k(\g,f))$ and so that $\int f=0$. 
\end{proof}

Let $Q_\al^W: V^k(\g_0)\otimes F(\g_\half)\rightarrow V^k(\g_0)\otimes F(\g_\half)$, ($\alpha\in\Pi$), be the derivation determined by 
\begin{equation}
\label{eq;2}
\begin{cases}{}Q_\al^W e_\bt=\sum\limits_{\gam\in[\al]} c^\gam_{\bt\al} \Phi_\gam,\quad Q_\al^W \Phi_\bt=(f|[e_\al,e_\bt]), & (\al\in\Pi_\half), \\
Q_\al^W e_\bt=(f|[e_\bt,e_\al]),\quad Q_\al^W \Phi_\bt=0,& (\al\in\Pi_1),
\end{cases} 
\end{equation}
\begin{equation}
\label{eq;1}
[Q_\al^W,\pd]=\frac{1}{k}\sum_{\begin{subarray}{c}\bt\in I\sqcup \Delta_0\\ \gam \in [\al]\end{subarray}} c_{\al,\bt}^\gam e_{\bar{\bt}} Q_\gam^W.
\end{equation}
\begin{theorem}
\label{main1}
For generic $k\in\C$, the classical affine $\W$-algebra $\W^k(\g,f)$ is isomorphic to the Poisson vertex subalgebra 
\begin{equation}\label{isom}
\W^k(\g,f) \cong \bigcap_{\al\in \Pi}\Ker\Bigl(Q_\al^W: V^k(\g_0)\otimes F(\g_\half)\rightarrow V^k(\g_0)\otimes F(\g_\half)\Bigr),
\end{equation}
of $V^k(\g_0)\otimes F(\g_\half)$ invariant under the  derivations $Q_\al^W$, $(\al\in \Pi)$.
\end{theorem}
\noindent We call the level $k\in\C$ \emph{generic} when \eqref{isom} holds.
\begin{proof}
Since $Q_i^W:V^k(\g_0)\otimes F(\g_\half)\rightarrow \L(-\al_i)\otimes F(\g_\half)$ acts by derivation, it decomposes as $Q_i^W=\sum_{\al\in[\al_i]} \vp^\al Q_\al^W$, where $Q_\al^W \in \Der(V^k(\g_0)\otimes F(\g_\half))$ and $\vp^\al$ is the multiplication by $\vp^\al$. We check that $Q_\al^W$ satisfies \eqref{eq;1} and \eqref{eq;2}.
By direct calculation, \eqref{eq;2} follows from the definition of $Q_i^W$. 
To show \eqref{eq;1}, recall that $V^k(\g_0)$ has a Virasoro element $L=\frac{1}{2k}\sum_{\al \in I\sqcup \Delta_0} e_\al e_{\bar{\al}}$, i.e., it satisfies 
$\{L_\lam L\}=(\lam+\pd)L+c_k L$ for some $c_k\in \C$ and $\pd=L_{(0)}$.
Then we have $[Q_i^W,\pd]=[\{Q_{i,\lam}^W-\}|_{\lam=0},\{L_{\mu}-\}|_{\mu=0}]=\{\{Q_{i\lam}^W L\}_\mu-\}|_{\lam=\mu=0}=\{-\pd Q_{i\lam}^W-\}|_{\lam=0}=0$. Here, we have used \eqref{Jacobi} in the second, \eqref{skew}in the third, and  \eqref{sesqui} in the last equality. Therefore,
$0=\sum_{\al\in [\al_i]} \vp^\al [Q_\al^W,\pd]-(\pd \vp^\al) Q_\al^W$.
Now \eqref{eq;1} follows from this by \eqref{eq;3}.
\end{proof}

\section{Geometric Realization of $\W^k(\g,f)$}
\subsection{Double coset space}
Consider the Lie algebra $\gpoly=\g\otimes \C[t^{\pm1}]$. By abuse of notation, we denote by $\k=\inv$ the invariant bilinear form on $\gpoly$ given by $(a t^n| b t^m)=(a|b) \delta_{n+m,0}$, which extends the one on $\g$. 
We extend the grading $\Gamma$ on $\g$ to $\gpoly$ by setting $\deg(a t^n)=\deg(a)+(d+1) n$ (see Section \ref{affW}) and fix a homogeneous basis $e_\al$, $(\al \in I \sqcup \afD)$, extending the basis $e_\al,$ $(\al\in I \sqcup \Delta)$, of $\g$. We denote $|\al|=\deg(e_\al)$ for simplicity and $c_{\al,\bt}^\gam$ the structure constants, i.e., $[e_\al,e_\bt]=\sum_{\gam}c_{\al,\bt}^\gam e_\gam$. Let $e_{\bar{\al}}$, $(\al \in I \sqcup \afD)$, denote the dual basis of $e_\al$, $(\al \in I \sqcup \afD)$, with respect to $\k$. 

Consider the completion $\glau=\varprojlim_{n>0} (\gpoly/\gpoly_{>n})$, which we call the \emph{loop algebra} of $\g$.
The grading on $\gpoly$ induces a decomposition $\glau=\lgp\oplus \glau_-$ where $\lgp=\varprojlim_{n>0} (\gpoly_{>0}/\gpoly_{>n})$ and $\glau_-=\gpoly_{\leq0}$. We decompose $X\in\glau$ as $X=X_++X_-\in \lgp\oplus \glau_{-}$. The subspace $\lgp$ is an affine scheme of infinite type. By setting $z_\al$ the coordinates of the basis $e_\al$, $(\al\in\afDp\subset \afD)$, of $\lgp$, we have $\C[\lgp]=\C[z_\al\mid \al\in \afDp]$.

Let $G$ be the connected simply-connected algebraic group whose Lie algebra is $\g$ and $\lG$ the loop group of $G$, whose Lie algebra is $\glau$. We have closed subgroups $\lG_{\pm}$ of $\lG$ whose Lie algebra is $\glau_{\pm}$.

Let us consider the quotient space $\lGn\backslash \lG$, which is an ind-scheme. It has an open subscheme $\lGp$. The $\C$-points of $\lGp$ is identified with the $\lGp$-orbits of the image of the identity in $\lGn\backslash \lG$. Since $\lGp$ is a prounipotent proalgebraic group, the exponential map $\exp: \lgp\rightarrow \lGp$ is an isomorphism of schemes. Thus the coordinate ring $\C[\lGp]$ is identified with $\C[\lgp]=\C[z_\al \mid \al\in\afDp]$.

The left multiplication $\lGp\times \lGp\rightarrow \lGp$, $((g_1,g_2)\mapsto g_1^{-1}g_2)$ induces a $\lgp$-action on $\C[\lGp]$ as derivations 
\begin{equation}\label{leftaction}
\xi^L: \lgp\rightarrow \Der(\C[\lGp]),\quad \phi\mapsto \phi^L.
\end{equation}
Since $\lGp$ is embedded into $\lGn\backslash \lG$ as an open subscheme, the right multiplication $\lGn\backslash \lG\times \lGn\rightarrow \lGn\backslash \lG$, $(([g_1],g_2)\mapsto [g_1g_2])$, induces a $\lgn$-action on $\C[\lGp]$ as derivations
\begin{equation}\label{rightaction}
\xi^R: \glau\rightarrow \Der(\C[\lGp]),\quad \phi\mapsto \phi^R.
\end{equation}
The commutator of these actions $\xi^L$, $\xi^R$ can be expressed by using the distinguished element $K=\exp(\sum_{\al\in\afDp} z_\al e_\al)\in\lGp(\C[\lGp])$. Note that this element is coordinate independent since it is identified with the identity morphism under the correspondence
$$\lGp(\C[\lGp])=\Hom(\C[\lGp],\C[\lGp]).$$ 
\begin{lemma}
\label{3.1}
For $u\in \lgp$ and $v\in \glau$, 
$[u^L,v^R]=[u,(K v K^{-1})_-]_+^L$
holds.
\end{lemma}
\begin{proof}
We use the notation $\e^X=\exp(X)$ for brevity. Let $\epsilon_i$ be dual numbers, i.e., satisfies $\epsilon_i^2=0$.
For $F\in\C[\lGp]$ and $g\in \lGp$, we have
\begin{align*}
(v^R F)(g)&= \text{the}\ \ep_2\text{-linear term of}\ F(g \e^{\ep_2 v})\\
&= \text{the}\ \ep_2\text{-linear term of}\ F((\e^{(g v g^{-1})_+})g),
\end{align*}
and 
\begin{align*}
(u^L v^R F)(g)&= \text{the}\ \ep_1\text{-linear term of}\ (v^RF)(\e^{-\ep_1 u}g)\\
&= \text{the}\ \ep_1\ep_2\text{-linear term of}\ F((\e^{(\e^{-\ep_1 u}g v g^{-1}\e^{\ep_1 u})_+})\e^{-\ep_1 u}g)\\
&= \text{the}\ \ep_1\ep_2\text{-linear term of}\ F(g_1(\ep_1,\ep_2)),
\end{align*}
where 
$$g_1(\ep_1,\ep_2)=\e^{-u \ep_1+(g v g^{-1})_+ \ep_2-([u,g v g^{-1}]_++(gv g^{-1})_+u) \ep_1 \ep_2}g.$$
Similarly, we obtain
$$(v^R u^L)(F)(g)=\text{the}\ \ep_1\ep_2\text{-linear term of}\ F(g_2(\ep_1,\ep_2))$$
where 
$$g_2(\ep_1,\ep_2)=\e^{-u \ep_1+(g vg^{-1})_+ \ep_2-u(g vg^{-1})_+ \ep_1\ep_2}g.$$
Since 
$$g_1(\ep_1,\ep_2)=\e^{-[u,(gvg^{-1})_-]_+\ep_1\ep_2}g_2(\ep_1,\ep_2),$$
we obtain $[u^L, v^R]F(g)=[u,(g v g^{-1})_-]_+^L F(g)$.
\end{proof}

In the sequel, we assume that $(\g,f)$ satisfies the condition (F) introduced in Section 1:
\begin{enumerate}
\item[(F1)] The grading $\Gamma$ is a $\Z$-grading.
\item[(F2)] There exists an element $y\in\g_d$ such that $s=f +y  t^{-1}\in\gpoly$ is semisimple.
\item[(F3)] The Lie subalgebra $\Ker(\ad_s)\subset\g[t^{\pm1}]$ is abelian and $\Im(\ad_s)\cap \gpoly_0=\g_0$.
\end{enumerate}
\begin{ex}
\emph{
The following pairs $(\g,f)$ satisfy (F).
\begin{enumerate}
\item 
$(\g,f)$ with $f$ principal nilpotent element. 
\item
$(\g(C_n),f_{(2^n)})$ and $(\g(C_{2n}),f_{(4^n)})$, $(n \geq2)$.
\item
$(\g(X_n),f)$ with $X=E,F,G$ listed in Table \ref{tb:exceptional}.
\end{enumerate}}
\end{ex}
Here we have used the classification of nilpotent elements of $\g(C_n)$ by symplectic partitions and that of $\g(X_n)$, $(X=E,F,G)$, by the weighted Dynkin diagram, (see \cite{collmcg93}.)
\begin{table}[h]{Table 4.1}
\label{tb:exceptional}
\begin{tabular}{|c|c||c|c|} \hline
    Dynkin diagram&weights of vertices& Dynkin diagram&weights of vertices \\ \hline
$G_2$         &\wtGa& $E_{6\ \text{(continued)}}$ & \wtEac \\ \cline{2-4}
\dynkinGtwo&\wtGb&$E_7$                 & \wtEba  \\ \cline{1-2}\cline{4-4}
$F_4$           &\wtFa&\dynkinEseven    & \wtEbb  \\ \cline{2-4}
\dynkinFfour&\wtFb& $E_8$                 & \wtEca \\ \cline{2-2}\cline{4-4}
                 &\wtFc&\dynkinEeight        & \wtEcb \\ \cline{2-2}\cline{4-4}
                  &\wtFd&                         &\wtEcc  \\ \cline{1-2}\cline{4-4}
$E_6$          &\wtEaa&                        &\wtEcd   \\ \cline{4-4}
 \dynkinEsix  &\wtEab&                          &\\ \hline

\end{tabular}
\end{table}

Let $\abel$ denote the completion of $\Ker(\ad_s)$ in $\glau$. It is abelian by (F2). Let $A\subset \lG$, (resp.\ $A_\pm\subset \lG$) be the closed subgroup of $\lG$ whose Lie algebra is $\abel$, (resp.\ $\abel_{\pm}=\abel\cap\glau_{\pm}$). 
We consider the quotient space $\hom$. It admits a left $\lGp$-action
$$\lGp \times \hom \rightarrow \hom,\quad (g, h\Ap) \mapsto g^{-1}h\Ap,$$
and a right $\An$-action 
$$\hom \times \An \rightarrow \lGn\backslash \lG/\Ap,\quad (h\Ap,g)\mapsto hg\Ap.$$
The right $\An$-action is well-defined since $A$ is abelian. These actions induce infinitesimal actions $\xi^L: \lgp\rightarrow \Der\C[\hom]$, $(\phi\mapsto \phi^L)$ and $\xi^R: \abeln\rightarrow \C[\hom]$, $(\phi\mapsto \phi^R)$ respectively. In particular, $\C[\lGp/A_L+]$ becomes a differential algebra by letting $s^R$ be the differential. 

Set 
\begin{equation}\label{E}
E_\al =k(e_\al|K s K^{-1})\in \C[\lGp],\quad (\al\in I \sqcup\Delta_0).
\end{equation}
We have $E_\al\in \C[\lGp/\Ap]$ since 
\begin{eqnarray*}
a^R (e_\al|K s K^{-1})&=& \text{the}\ \ep\text{-linear term of}\ (e_\al|K \e^{\ep a} s (K \e^{\ep a})^{-1})\\
&=&(e_\al|K[a,s]K^{-1})=0,
\end{eqnarray*}
for $a\in \abelp$.
Since the Poisson vertex algebra $V^k(\g_0)$ is an algebra of differential polynomials, there exists a unique homomorphism of differential algebras
$$\Psi_k: V^k(\g_0)\rightarrow \C[\lGp/\Ap],\quad e_\al \mapsto E_\al.$$
\begin{theorem}
\label{main2}
Suppose that $(\g,f)$ satisfies (F) and $k\in\C$ is generic.
Then $\Psi_k$ is an isomorphism of differential algebras and satisfies 
\begin{equation}\label{intertwins}
\Psi_k Q_\al^W \Psi_k^{-1}=-\frac{1}{k}e_\al^L,\quad \al\in\Pi.
\end{equation}
\end{theorem}
\begin{proof}
Since $V^k(\g_0)$ and $\C[\hom]$ are polynomial rings, it suffices to show that the linear map 
$d_0\Psi_k: T^*_0\text{Spec}V^k(\g_0) \rightarrow T^*_{[e]}\hom$ 
between the cotangent spaces is an isomorphism. By the identification $\lgp^* \cong \glau_{<0}$ and $\abelp^* \cong \abeln$ induced by $\k$, we obtain  $T_{[e]}^*\hom \cong (\lgp/\abelp)^* \cong \glau_{<0}/\abeln$. This isomorphism is given by $dz_\al \mapsto e_{\bar{\al}}$. Under this isomorphism, we have $d_0\Psi_k(\pd^ne_\al)=k\ad_s^{n+1}  e_\al$, as proved by inductive use of the $n=0$ case:
\begin{align*}
d_0\Psi_k(e_\al)&=kd(e_\al|KsK^{-1})=kd(e_\al|[\sum_{|\bt|=1}z_\bt e_\bt,s])\\
&=\sum\limits_{|\bt|=1}k([s,e_\al]|e_\bt) dz_\bt=k\sum\limits_{|\bt|=1}([s,e_\al]|e_\bt) e_{\bar{\bt}}\\
&=k[s,e_\al].
\end{align*} 
It follows from (F2) and (F3) that $\Psi_k$ is an isomorphism.

To show \eqref{intertwins}, it suffices to show 
$$u_\al^L E_\bt=(f|[e_\bt, e_\al]),\ [u_\al^L,s^R]=\frac{1}{k}\sum_{\begin{subarray}{c}\rho\in I\sqcup \Delta_0\\ \gam\in[\al]\end{subarray}} c_{\al,\rho}^\gam E_{\bar{\rho}} u_\gam^L,\quad 
\al\in \Pi_1,\ \bt\in I\sqcup \Delta_0,$$
where $u_\al=-\frac{1}{k}e_\al$ by \eqref{eq;1}, \eqref{eq;2}.
For the first one, we have
\begin{align*}
u_\al^L E_\bt&= \text{the}\ \ep \text{-linear term of}\ k(e_\bt|\e^{-\ep e_\al/k}K s K^{-1}\e^{\ep e_\al/k})\\
&=k(e_\bt|[-\frac{1}{k}e_\al,KsK^{-1}])=([e_\bt,e_\al]|KsK^{-1})\\
&=([e_\bt,e_\al]|s)=([e_\bt,e_\al]|f).
\end{align*}
The second one follows from Lemma \ref{3.1} since
$$[u_\al^L,s^R]=[u_\al,(K s K^{-1})_0]_+^L=\frac{-1}{k^2}\sum\limits_{\begin{subarray}{c}\rho\in I\sqcup \Delta_0\\ \gam\in[\al]\end{subarray}} c_{\al,\rho}^\gam E_{\bar{\rho}} e_\gam^L=\frac{1}{k}\sum_{\begin{subarray}{c}\rho\in I\sqcup \Delta_0\\ \gam\in[\al]\end{subarray}} c_{\al,\rho}^\gam E_{\bar{\rho}} u_\gam^L.$$
\end{proof}
Note that the Lie subalgebra $\g_+=\oplus_{j>0}\g_j$ is generated by the subspace $\oplus_{\al\in\Pi}\g_\al$. Let $G_+\subset \lGp$ denote the closed subgroup whose Lie algebra is $\g_+$. 
\begin{cor}\label{doublecoset}
The isomorphism $\Psi_k$ restricts to an isomorphism  of differential algebras $\W^k(\g,f)\cong \C[\dcoset]$.
\end{cor}
\begin{proof}
The claim follows from Theorem \ref{main2} since 
$$\W^k(\g,f)\cong \bigcap_{\al\in\Pi_1} \Ker\bigl(Q_\al^W:V^k(\g_0)\rightarrow V^k(\g_0)\bigr),$$
by Theorem \ref{main1} and (F1).
\end{proof}
\subsection{Mutually commutative derivations on $\W^k(\g,f)$}
By (F3), we have a Lie algebra homomorphism $\xi^R:\abeln \rightarrow \Der(\C[\hom]),\ (a \mapsto a^R)$.
\begin{proposition}
\sl
The action $\xi^R$ preserves $\C[\dcoset]\subset\C[\hom]$. 
\end{proposition}
\begin{proof}
By Lemma \ref{3.1}, we have, for $e_\al\in\g$, $(\al\in\Pi_1)$, and $a\in \abeln$, 
$$[e_\al^L,a^R]=[e_\al,(K a K^{-1})_-]_+^L=[e_\al,(K a K^{-1})_0]^L=\sum\limits_{\bt \in I\sqcup \Delta_0} F_\bt[e_\al,e_\bt]^L,$$
for some polynomials $F_\bt\in \C[\hom]$. 
Hence, for $G\in \C[\dcoset]$, we have
$$e_\al^L(a^R G)=[e_\al^L,a^R] G=\sum\limits_{\bt \in I\sqcup\Delta_0} F_\bt[e_\al,e_\bt]^L G=0$$.
\end{proof}
Since $\abeln$ is abelian, its image $\xi^R(\abeln)\subset \Der\C[\dcoset]$ is also abelian. By Theorem \ref{main2}, $\xi^R(\abeln)$ is identified with a set of commutative derivations of $\W^k(\g,f)$, which we denote by $\mathcal{H}^k(\g,f)$. In the next section, we will prove that $\mathcal{H}^k(\g,f)$ is an integrable Hamiltonian hierarchy associated with the classical affine $\W$-algebra $\W^k(\g,f)$.
\section{$\mathcal{H}^k(\g,f)$ as an integrable Hamiltonian hierarchy}
In this section, we always assume that the condition (F)  holds and $k\in\C$ is generic, and identify $V^k(\g_0)$ with $\C[\hom]$ by $\Psi_k$, (Theorem \ref{main2}).
\subsection{Construction of Hamiltonians}
Let $\Omega_\dr(\lGp)=\C[\lGp]\otimes \bigwedge{\lgp^*}$ denote the algebraic de Rham complex of $\lGp$. 
Here $\lgp^*$ is the vector space dual to the space $\lgp$ of the right invariant vector fields $\lgp^L$ on $\lGp$. We take a basis $\vp^\al$, $(\al\in\afDp)$, of $\lgp^*$ so that $\vp^\al(e_\bt)=\delta_{\al,\bt}$, $(\al,\bt\in \afDp)$, holds. 
Note that the complex $\Omega_\dr(\lGp)$ coincides with the Chevalley-Eilenberg complex of the Lie algebra $\lgp$ with coefficients in $(\C[\lGp],\xi^L)$.
Similarly, let $\Omega_\dr(\Ap)=\C[\Ap]\otimes \bigwedge{\abelp^*}$ denote the algebraic de Rham complex of $\Ap$. Then the inclusion $\Ap \hookrightarrow \lGp$ induces the projection $\Omega_\dr(\lGp) \rightarrow \Omega_\dr(\Ap)$. It restricts to $\abelp^R$-invariant subcomplexes: 
\begin{equation}
\label{eq;4}
\pi: \C[\hom]\otimes \bigwedge{\lgp^*} \rightarrow \C \otimes \bigwedge{\abelp^*}.
\end{equation}
As a $\lgp$-module, $\C[\hom]$ is isomorphic to the $\lgp$-module $\text{Coind}_{\abelp}^{\lgp}\C$ coinduced from the trivial $\abelp$-module $\C$. Then $\pi$ induces an isomorphism
$$H^*(\lgp; \C[\hom]) \cong H^*\bigl(\lgp; \text{Coind}_{\abelp}^{\lgp} \C\bigr) \cong H^*(\abelp;\C)$$
by Shapiro's lemma, (cf. \cite{fuks86}). The right hand side is isomorphic to $\bigwedge \abelp^*$ since $\abelp$ is abelian. The action $\xi^R$ of $\abeln$ induces $\abeln$-actions on the complexes $\Omega_\dr(\lGp)$, $\Omega_\dr(\Ap)$, which  commute with $\pi$. Since $\abel$ is abelian, $\abeln$ acts on $\Omega_\dr(\Ap)$ trivially. Thus $\abeln$ acts on the cohomology $H^*(\lgp; \C[\hom])$ trivially. In particular, $s$ acts on $H^*(\lgp; \C[\hom])$ trivially. 
By abuse of notation, we write the above $\abeln$-actions by $\xi^R$. 

Consider the double complex $\mathscr{C}$
$$\begin{CD}
\C @>\iota>> \Omega_\dr(\lGp)  @>\pm s^R>> \Omega_\dr(\lGp)  @>\ep>> \C.
\end{CD}$$
Here $\iota$ is the unit morphism and $\ep$ the counit morphism.
It has the following shape:

$$\begin{CD}
@.         \C_{(0,2)}    \\
@.       @A\ep AA    \\
0_{(-1,1)} @>>> \C[\hom]_{(0,1)} @>d>> \C[\hom] \otimes \glau^*_{>0\ (1,1)} @>d>>\cdots\\
@.       @A \pd=s^R AA    @A -s^R AA  \\
0_{(-1,0)} @>>> \C[\hom]_{(0,0)} @>d>> \C[\hom] \otimes \glau_{>0\ (1,0)} @>d>>\cdots,\\
@.       @A\iota AA \\
@.         \C_{(0,-1)}
\end{CD}$$
where the subscript $(i,j)$ denotes the bidegree of $\mathscr{C}$. Then the calculation of the cohomology $H^*(\mathscr{C})$ via spectral sequences gives the isomorphisms
\begin{equation}
\abelp^* \cong H^1(\mathscr{C}) \cong \Ker\Bigl(d: \frac{\C[\hom]}{\C \oplus \Im(s^R)} \rightarrow \frac{\C[\hom] \otimes \lgp^*}{\Im(s^R)}\Bigr).
\end{equation}
By Theorem \ref{main2}, we have an isomorphism $\C[\hom]/\C \oplus \Im(s^R)\cong \Lie(V^k(\g_0))/\C$ as vector spaces.
Since $V^k(\g_0)$ is an algebra of differential polynomials, $\int: \C\rightarrow \Lie\bigl(V^k(\g_0)\bigr)$ is injective.
Thus we may lift an element $\int f\in \Lie(V^k(\g_0))/\C$ to $\int \tilde{f}\in \Lie(V^k(\g_0))/\C$ which is without the constant term $\tilde f(0)=0$.
Identifying $\abeln$ with $\abelp^*$ by $\k$, we obtain an isomorphism 
\begin{equation*}
\abeln\cong \abelp^* \cong \left\{\int f\in\Lie(V^k(\g_0))| f(0)=0,\ d\int f=0\right\},\quad a\mapsto \int H(a).
\end{equation*}
The statement of the following proposition makes sense by Lemma \ref{2.3}.
\begin{proposition}
\label{4.1}
\sl
The image $\int H(\abeln)$ lies in $\Lie(\W^k(\g,f))$.
\end{proposition}
\begin{proof}
We may assume that the basis $\{e_\al\}_{\al\in\afDp}$ of $\lgp$ respects the decomposition $\lgp=\Im(\ad_s)\oplus \Ker(\ad_s)$. 
For $a\in \abeln$, the element $\int H(a)\in \Lie(V^k(\g_0))$ satisfies, by construction, $d H(a)=s^R A$ for some element $A=\sum_{\al\in\afDp}F_\al \otimes \vp^\al \in \C[\hom]\otimes\lgp^*$.
Then, 
\begin{align*}
d H(a)&=s^R A= s^R\sum\limits_{\al\in\afDp}F_\al\otimes \vp^\al\\
&=\sum\limits_{\al\in\afDp} \pd(F_\al) \otimes \vp^\al +\sum\limits_{\al\in\afDp} F_\al \otimes s^R(\vp^\al)\\
&=\sum\limits_{\al\in\Pi_1}
\left(\pd F_\al +\sum_{\begin{subarray}{c}\bt\in\Pi_1\\ \gam\in \sqcup \Delta_0\end{subarray}}F_\bt c_{\al,\gam}^{\bt}E_{\bar{\gam}} \right) 
\otimes \vp^\al+ \sum\limits_{\al\in\afDp\backslash\Pi_1}\widetilde{F_\al^n}\otimes \vp^\al,
\end{align*}
for some element $\widetilde{F_\al^n}\in\C[\hom]$.
In the last equality, we have used Lemma \ref{4.2} below. By the identification $\C[\hom]\cong V^k(\g_0)$ and $\C[\hom] \vp^\al \cong \L(-\al)$, $(\al\in \Pi_1)$, we obtain
$$\sum\limits_{\al\in\Pi_1}\left(\pd F_\al +\sum_{\begin{subarray}{c}\bt\in\Pi_1\\ \gam\in I\sqcup\Delta_0\end{subarray}}F_\bt c_{\al,\gam}^{\bt}E_{\bar{\gam}} \right) \otimes \vp^\al= \pd\left(\sum\limits_{\al\in\Pi_1}F_\al \otimes \vp^\al\right)=0\in \bigoplus_{\al\in \Pi_1}\Lie\bigl(\L(-\al)\bigr).$$
On the other hand, by the definition of $d$, we have 
$d H(a)= \sum\limits_{\al\in\afDp} e_\al^L H(a) \otimes \vp^\al$. 
It follows that $e_\al^L \int H(a)=0$ for all $\al\in\Pi_1$, which implies $\int H(a)\in\Lie(\W^k(\g,f))$ by Theorem \ref{main2}. 
\end{proof}
\begin{lemma}
\label{4.2}
\sl
The following formula holds
\begin{equation}\label{eq;4.2}
s^R \vp^\al=\frac{1}{k}\sum_{\begin{subarray}{c}\bt\in\afDp\\ \gam\in I\sqcup \Delta_0\end{subarray}} E_{\bar{\gam}} c_{\bt,\gam}^\al \vp^\bt+\sum_{\bt\in \afDp} c_{\bt,s}^\al \vp^\bt,\quad \al\in \afDp.
\end{equation}
\end{lemma} 
\begin{proof}
By Lemma \ref{3.1}, we have 
\begin{eqnarray}
\label{eq;5}
[s^R, e_\al^L]=-[e_\al, (K s K^{-1})_-]_+^L=-\frac{1}{k}\sum_{\bt\in I\sqcup \Delta_0} E_{\bar{\bt}} [e_\al, e_\bt]^L - [e_\al, s]_+^L.
\end{eqnarray}
Then, from the definition of $\lgp^*$, we have
$$\left(s^R \vp^\al\right) (e_\bt)=-\vp^\al(\pd e_\bt)=\frac{1}{k}\sum_{\gam\in I\sqcup \Delta_0}E_{\bar{\gam}}c_{\bt,\gam}^\al+c_{\bt,s}^\al.$$
The claim \eqref{eq;4.2} follows from it immediately.
\end{proof}
\subsection{Poisson vertex superalgebra structure on $\C[\hom] \otimes \bigwedge \lgp^*$}
We extend the Poisson vertex algebra structure on $\C[\hom]$ given by Theorem \ref{main2}, to the whole differential superalgebra $(\C[\hom] \otimes \bigwedge \lgp^*,s^R)$.
\begin{proposition}
\label{4.3}
\sl
The differential superalgebra $(\C[\hom]\otimes\bigwedge\lgp^*,s^R)$ admits a unique Poisson vertex superalgebra structure, which satisfies
\begin{enumerate}
\item $\{ u_\lam v\}= [u, v]+k(u|v) \lam,\quad (u,v\in \g_0)$,
\item $\{\vp^\al_\lam\vp^\bt\}=0,\quad (\vp^\al, \vp^\bt \in \glau_1^*\cup \abelp^*)$,
\item $\{u_\lam \vp^\al\}=-\sum c_{u,\bt}^\al \vp^\bt,\quad (\vp^\al\in \glau_1^*\cup \abelp^*,u\in \g_0)$.
\end{enumerate}
It satisfies
\begin{eqnarray}
\label{eq;6}
\{u_\lam \vp^\al\}=-\sum c_{u,\bt}^\al \vp^\bt,\quad (\vp^\al\in \glau^* ,u\in \g_0).
\end{eqnarray}
\end{proposition}
\begin{proof} 
It follows from Lemma \ref{4.2} that $\C[\hom] \otimes \bigwedge \lgp^*$ is an algebra of differential polynomials in the variables given by the union of bases of $\g_0$ and $\glau_1^* \cup \abelp^*$. By Theorem \ref{1.2}, it suffices to show \eqref{skew}, \eqref{Jacobi} for these variables in order to prove that  (1)-(3) define a Poisson vertex superalgebra on $\C[\hom] \otimes \bigwedge \lgp^*$. The only non-trivial ones are \eqref{Jacobi} for $\{ e_\al, e_\bt, \vp^\gam\}$. By \eqref{skew}, they reduce to the special case 
\begin{equation}\label{Jacobicheck}
\{e_{\al\lam}\{e_{\bt\mu}\vp^\gam\}\}-\{e_{\bt\mu}\{e_{\al\lam}\vp^\gam\}\}=\{\{e_{\al\lam}e_\bt\}_{\lam+\mu}\vp^\gam\}.
\end{equation}
Since it coincides with the Jacobi identity $(\coad(e_\al)\coad(e_\bt)-\coad(e_\bt)\coad(e_\al))\vp^\gam=\coad([e_\al,e_\bt])\vp^\gam$ of the coadjoint $\g_0$-action on $\lgp^*$, \eqref{Jacobicheck} holds. 

Next, we show \eqref{eq;6}.
Since $\ad_s$ is an isomorphism on $\Im(\ad_s)\subset \glau$, we have its inverse $\ad_s^{-1}$ on $\Im(\ad_s)$ and on its dual space $\Im(\ad_s)^*$. Then we have a decomposition $\glau_{>0}=\ad_s^{-1}\Im(\ad_s)_{>0}\oplus (\glau_1\cup \abelp)$. We show \eqref{eq;6} for $\vp^\al\in \Im(\ad_s)$ by induction on degree. From Lemma \ref{4.2}, we have
$$\ad_s^{-1}\vp^a= \pd \vp^a -\frac{1}{k} \sum_{\al,b} c_{b,\al}^a e_{\bal} \vp^b,\quad \vp^\al\in \Im(\ad_s)_1.$$  
Here the sum with respect to Greek letters (resp.\ Roman letters) is taken over $I\sqcup \Delta_0$ (resp.\ $\afDp$). We use the same rule below. Then, for $e_\al\in\g_0$ and $\vp^a\in \Im(\ad_s)_1$, we have 
$$\{e_{\al\lam}(\ad_s^{-1})^n\vp^a\}=-\sum\limits_{|b|=|a|} c_{\al,b}^a (\ad_s^{-1})^n\vp^b,$$
which is proved by the inductive use of the $n=1$ case:
\begin{align*}
\{e_{\al\lam}\ad_s^{-1}\vp^a\}
&= \{e_{\al,\lam}\pd \vp^a\}-\frac{1}{k} \sum_{\bt,b} c_{b,\bt}^a \{e_{\al\lam}e_{\bbt} \vp^b\}\\
&=(\pd+\lam)\{e_{\al,\lam}\pd \vp^a\}-\frac{1}{k}\sum_{\bt, b} c_{b,\bt}^a( \{e_{\al \lam}e_{\bbt}\}\vp^b+\{e_{\al\lam}\vp^b\}e_{\bbt})\\
&=-(\pd+\lam)\sum_b c_{\al b}^a \vp^b-\frac{1}{k} \sum\limits_{\al,b}
c_{b,\al}^a\left(([e_\al, e_{\bbt}]+k(e_\al|e_{\bbt})\lam) \vp^b-\sum_c c_{\al,c}^b e_{\bbt}\vp^c\right)\\
&=-\lam\left(\sum_b c_{\al,b}^a \vp^b+\sum_{\bt,b} c_{b,\bt}^a (e_\al|e_{\bbt})\vp^b\right)\\
&\qquad-\left(\sum\limits_b c_{\al,b}^a\pd \vp^b+\frac{1}{k}\sum_{\bt,b}c_{b,\bt}^a\bigl([e_\al, e_{\bbt}]\vp^b-\sum_{c}c_{\al,c}^b e_{\bbt}\vp^c\bigr)\right)\\
&=-\sum_b c_{\al,b}^a \ad_s^{-1}\vp^b.
\end{align*}
In the last equality, we used $(e_\al|e_{\bbt})=\delta_{\al,\bt}$ for the $\lam$-linear term and use Lemma \ref{4.2}, $c_{a,b}^c=c_{\bar{c},a}^{\bar{b}}$, and the Jacobi identity of the Lie bracket for the constant term.
This implies \eqref{eq;6} for $\vp^\al\in \Im(\ad_s)^*$. Then \eqref{eq;6} follows from this and (3). 
\end{proof}
The de Rham differential $d=\sum_{\al \in \afDp} e_\al^L \otimes \vp^\al$ on $\C[\hom]$ can be described in term of the Poisson vertex superalgebra structure in Proposition \ref{4.3}. 
\begin{lemma}
\label{4.4}
\sl 
The differential $d$ is determined uniquely by\\
(i) $d E_\bt=-k\sum_{\al\in\afDp}([e_\bt,s]|e_\al) \otimes \vp^\al,\quad (\bt\in I\sqcup \Delta_0)$,\\
(ii) $[d, \pd]=0$.
\end{lemma}
\begin{proof}
(1) follows immediately from the formula $e_\bt^L E_\al=-k([e_\al,s]|e_\bt)$ (see also \eqref{E}). (2) follows from \eqref{eq;5} and Lemma \ref{4.2}.
\end{proof}
Let $s^*\in \lgp^*$ denote the element corresponding to $s\in\lgn$ by the identification $\kappa=\inv: \lgn\cong \lgp^*$ and $\bar{s}\in \lgp$ the element corresponding to $s^*$.
\begin{proposition}
\sl
Under the isomorphism $\Psi_k: V^k(\g_0)\cong \C[\hom]$, 
$$d=-k\{s^*_\lam -\}|_{\lam=0},$$
holds on $\C[\hom]$.
\end{proposition}
\begin{proof}
It suffices to show that $-k\{s^*_\lam -\}|_{\lam=0}$ satisfies (i), (ii)  in Lemma \ref{4.4}. (i) follows from the definition of $s^*$ and Proposition \ref{4.3}, (3). (ii) follows from \eqref{sesqui}.
\end{proof}
The following property of the derivation $\eta_a=\eta(\int H(a))\in\Der(\C[\hom])$, $(a\in\abeln)$, will be used.
\begin{lemma}
\label{4.5}
\sl
For $a\in\abeln$, there exist some polynomials $F_\al(a)\in\C[\hom]$, ($\al\in I \sqcup\Delta_0$), which satisfies
$$[u^L,\eta_a]=\sum_{\al\in I \sqcup\Delta_0} F_\al(a) [u, e_\al]^L,\quad u\in\glau_1.$$
\end{lemma}
\begin{proof}
By construction of $\int H(a)$, we have $d \int H(a)=\int dH(a)=0\in\Lie(\C[\hom]\otimes\bigwedge \lgp^*)$ and thus $\{dH(a)_\lam-\}|_{\lam=0}=0$.
By \eqref{Jacobi}, we have
\begin{align*}   
\{dH(a)_\lam -\}|_{\lam=0}
&=\Bigl[\{s^*_\lam-\}|_{\lam=0}, \{H(a)_\mu-\}|_{\mu=0}\Bigr]
=\Bigl[\sum_{\al\in\afDp} e_\al^L\otimes\vp^\al, \eta_a\Bigr]\\
&=\sum_{\al\in\afDp} \Bigl([e_\al^L, \eta_a]\otimes \vp^\al -e_\al^L\otimes\eta_a(\vp^\al)\Bigr).
\end{align*}
It follows from them that
$$[e_\al^L, \eta_a]=\sum_{\bt\in\afDp} \langle e_\al \mid \eta_a(\vp^\bt) \rangle e_\bt^L,$$
where $\langle-\mid- \rangle: \lgp\times\lgp^*\rightarrow\C$ denotes the canonical pairing.
By Proposition \ref{4.3}, we have
$$\eta_a(\vp^\bt)= \{ H(a)_\lam \vp^\bt\}|_{\lam=0}=\sum\limits_{\gam\in I\sqcup \Delta_0} \{ e_{\gam\pd} \vp^\bt\}_\rightarrow \frac{\delta H(a)}{\delta e_\gam}
=-\sum\limits_{\begin{subarray}{c}\gam\in I\sqcup \Delta_0\\ \rho\in \afDp\end{subarray}} c_{\gam,\rho}^\bt\vp^\rho \frac{\delta H(a)}{\delta e_\gam}.$$
Hence, for $\al \in \glau_1$ we obtain 
\begin{align*}
[e_\al^L, \eta_a]&=\sum_{\bt\in \afDp} \langle e_\al \mid \eta_a(\vp^\bt)\rangle e_\bt^L
=-\sum\limits_{\bt\in \afDp}\sum\limits_{\begin{subarray}{c}\gam\in I\sqcup \Delta_0\\ \rho\in \afDp\end{subarray}}
c_{\gam,\rho}^\bt \frac{\delta H(a)}{\delta e_\gam}
\langle e_\al \mid \vp^\rho\rangle e_\bt^L\\
&=\sum\limits_{\gam\in I\sqcup \Delta_0} \frac{\delta H(a)}{\delta e_\gam} [e_\al, e_\gam]^L.
\end{align*}
This completes the proof.
\end{proof}
\subsection{Integrable Hamiltonian hierarchy} 
Let $\Vect(\lGp)$ denote the Lie algebra of the algebraic vector fields on $\lGp$. Recall the infinitesimal $\lgp$-actions \eqref{leftaction}, \eqref{rightaction}.
\begin{lemma}
\label{4.6}
\sl
We have $\Vect(\lGp)^{\lgp^L} =\lgp^R$ and $\Vect(\lGp)^{\lgp^R} =\lgp^L$.
\end{lemma}
\begin{proof}
Although this is well-know, we give a proof for the completeness of the paper.
The pairing 
$$\U(\lgp) \times \C[\lGp]\rightarrow \C,\quad (X_1X_2\cdots X_n, F)\mapsto (X_1^LX_2^L\cdots X_n^LF)(e),$$ defined by
is $\lgp$-invariant and nondegenerate. Here $\U(\lgp)$ denotes the universal enveloping algebra of $\lgp$. Hence we have $\C[\lGp]^{\lgp^L} \cong \U(\lgp)/\lgp\U(\lgp) \cong \C$. Since $\lgp^L$ commutes with $\lgp^R$ and $\Vect(\lGp)\cong \C[\lGp]\otimes_{\C}\lgp^R$, we obtain
$$\Vect(\lGp)^{\lgp^L}\cong (\C[\lGp]\otimes_{\C} \lgp^R)^{\lgp^L} = \C[\lGp]^{\lgp^L}\otimes_{\C} \lgp^R\cong \lgp^R.$$
The latter claim is proved similarly.
\end{proof}
We say that an element $X\in\Vect(\lGp)$ satisfies (P) if  
$$[u^L,X]\in\sum_{\al\in I \sqcup\Delta_0}\C[\lGp][u, e_\al]^L,\quad u\in\glau_1.$$ 
Let $\liep$ denote the set of elements in $\Vect(\lGp)$ satisfying (P). It is straightforward to show that $\liep\subset \Vect(\lGp)$ form a Lie subalgebra.
The Witt algebra $\witt=\C((t)) \frac{\pd}{\pd t}$ acts on $\glau$ by derivations with respect to $t$. Since the subalgebra $\witt_-=\C[t^{-1}]t\frac{\pd}{\pd t}$ preserves $\glau_{\leq0}$, it acts on $B\backslash G((t))$ infinitesimally and therefore on the open dense subset $\lGp$. The induced $\witt_-$-action on $\C[\lGp]$ is given by
$$L_n \exp\left(\sum_{\al\in\afDp} z_\al e_\al\right)= \text{the}\ \ep\text{-linear term of}\ \exp\left(\sum_{\al\in\afDp} z_\al e_\al\right)e^{\ep L_n},$$
where $L_n=-t^{n+1}\frac{\pd}{\pd t}\in \witt_-$.
\begin{lemma}
\sl
We have $\liep=\witt_-\ltimes \glau^R$.
\end{lemma}
\begin{proof}
By Lemma \ref{3.1}, we have $\glau^R\subset \liep$.  In particular, this implies that $\liep$ is a $\glau$-module.
The inclusion $\witt_-\subset \liep$ is shown in the same way as in the proof of Lemma \ref{3.1}.
It is obvious that these two actions give an action of their semidirect product $\witt_-\ltimes \glau^R\subset \liep$.

To show its equality, consider the quotient $\glau$-module $M:=\liep/\witt_- \ltimes \glau^R$. 
Let us show that $M$ belongs to the category $\mathcal{O}$ of the affine Lie algebra  $\afg=\gpoly \oplus \C C \oplus \C L_0$ at level 0. Let $\afg=\afn_+\oplus \afh \oplus \afn_-$ be the standard triangular decomposition of $\afg$.
Indeed, the action of the Cartan subalgebra $\widetilde{\h}=\h \oplus \C C \oplus L_{0}$ on $\Vect(\lGp)$ is diagonalizable and, moreover, preserves $\liep$ and $\witt_- \ltimes \glau^R$. Hence it is diagonalizable on $M=\liep/\witt_- \ltimes \glau^R$. 
Let us show that the dimension $\dim \U(\afn_+) w$ is finite for any  $w\in M$. Let $\tw\in\liep$ denote a lift of $w$. 
The decomposition $\afn_+=\g_0^+\ltimes\afn_{>0}$ where $\g_0^+=\bigoplus_{\al \in \Delta_{0+}} \g_{0\al}$ and $\afn_{>0}=\afn_+\cap \glau_{>0}$ induces a decomposition of the universal enveloping algebra $\U(\afn_+) \cong \U(\g_0^+)\otimes\U(\afn_{>0})$ as vector spaces. By definition of $\liep$, $\tw$ satisfies, for any $u\in\glau_1$, 
$$[u^L, w]= \sum\limits_{\al\in I\sqcup\Delta_0} F_\al(w) [e_\al, u]^L\ \text{for some element}\ F_\al(w)\in\C[\lGp].$$
Since $\afn_{>0}^L$ and $\afn_{>0}^R$ commute, for any $Z\in \U(\afn_{>0})$,
$$ [u^L,Z^R(w)]=Z^R([u^L,w])=\sum\limits_{\al \in I \sqcup\Delta_0} Z^R\bigl(F_\al (w)\bigr) [e_\al, u]^L.$$
If $Z^R(F_\al(w))=0$ for all $\al\in I \sqcup\Delta_0$ and $u\in \glau_1$, then by Lemma \ref{4.6}, we have $[Z^R, w]\in\afn_{>0}^R$, which is zero in $M$. 
Considering the $\widetilde{\h}$-weights of $Z^R(F_\al(w))$, we conclude that $\dim\U(\afn_{>0})^R(F_\al(w))<\infty$ for each $\al$ and hence $\dim \U(\afn_{>0})^R w<\infty$. 
Thus it remains to show $\dim \U(\g_0^+) w<\infty$. 
Take $u\in\afg_1$ and $v\in \g_0^+$, we have $[u^L, v^R]=[u, v]^L$ by Lemma \ref{3.1}. 
Then by setting $[u^L,w]=\sum_{\al\in I\sqcup \Delta_0}G_\al(w)[u,e_\al]^L$ for some element $G_\al\in\C[\lGp]$, we have
\begin{align*}
[u^L, [v^R, w]]&=[[u, v]^L, w]+ [v^R, [u^L, w]]\\
&=\sum\limits_
{\al\in I\sqcup \Delta_0} G_\al(w)[[u, v], e_\al]^L +\sum\limits_{\al\in I\sqcup \Delta_0} v^R(G_\al(w)) [u, e_\al]^L.
\end{align*}
By induction, we obtain
\begin{align*}
[u^L,[v_1^R,[v_2^R,\cdots,[v_n^R,w]\cdots]
=\sum_{\begin{subarray}{c}\al\in I\sqcup \Delta_0\\A\sqcup B=\{1,2,\cdots,n\}\end{subarray}} (v_A^R G_\al(w)) [[u,v_B],e_\al]^L
\end{align*}
for $v_1,v_2,\cdots,v_n\in \g_0^+$, where $v_A^R=v_{a_1}^Rv_{a_2}^R\cdots v_{a_{\mathrm{max}}}^R$ for $A=\{a_1,a_2,\cdots,a_{\mathrm{max}}\}$, $(a_1<a_2<\cdots <a_{\mathrm{max}})$, and $[u,v_B]=[\cdots[u,v_{b_1}],v_{b_2}],\cdots],v_{b_{\mathrm{max}}}]$ for $B=\{b_1,b_2,\cdots,$ $b_{\mathrm{max}}\}$, $(b_1<b_2<\cdots<b_{\mathrm{max}})$.  
Since $\ad_{\g_0^+}^N \glau_1=0$ and $((\g_0^+)^R)^N G_\al(w)=0$ for $N$ large enough, it follows that
$[u^L,[v_1^R,[v_2^R,\cdots,[v_n^R,w]\cdots]=0$ 
for $n$ sufficiently large and hence $[v_1^R,[v_2^R,\cdots,[v_n^R,w],\cdots]\in\afn_{>0}^R$, which is zero in $M$. 
Therefore, we obtain $\dim\U(\g_0^+)^R w<\infty$. Thus we conclude that $M$ belongs to the category $\mathcal{O}$.

Suppose $M\neq0$. Then it contains a nonzero highest weight vector $w$. Let $\tw\in\liep$ denote a lift of $w$. Then $\witt_- \ltimes \glau^R + \C \tw$ is an extension of $\afn$-module $\witt_- \ltimes \glau^R$ by a trivial 1-dimensional $\afn$-module $\C w$, which is nontrivial by construction.
Thus this extension defines a nonzero element in $H^1(\afn; \witt_- \ltimes \glau^R)$. Since the action $\glau^R$ on $\C[\lGp]$ is faithful, $\glau^R\cong \glau$ as $\afn$-modules. 
On the other hand, the cohomology $H^n(\afn; \witt_- \ltimes \glau^R)$ is describes as 
$$H^i(\afn;  \witt_- \ltimes \glau) \cong H^{i-1}(\afn; \C) \otimes \witt_{>0},$$
where $\witt_{>0}=\C[[t]]t^2 \frac{\pd}{\pd t}$ (\cite{feiginfrenkel95}). This isomorphism is $\widetilde{\h}$-equivariant. In particular, all the degrees of the homogeneous elements with respect to the action of $x+(d+1) L_0$ in the first cohomology $H^1(\afn;  \witt_- \ltimes \glau) \cong \witt_{>0}$ are positive. (See Section 1 for the definition of $x$ and $d$.)
On the other hand, the cocycle corresponding to $\C w$ is negative due to the definition of $\liep$, which is a contradiction. Therefore, we conclude $M=0$. This completes the proof 
\end{proof}
\begin{theorem}
\label{main3}

Suppose that $(\g,f)$ satisfies (F) and $k\in\C$ is generic. Then, $\mathcal{H}^k(\g,f)$ is an integrable Hamiltonian hierarchy associated with $\W^k(\g,f)$.
\end{theorem}
\begin{proof}
By abuse of notation, let $\eta_a\in \Der(V^k(\g_0))$ also denote the corresponding element in $\Vect(\hom)$ by $\Psi_k$ in Theorem \ref{main2}. Since $\lGp \cong (\hom) \times \Ap$ as affine schemes in an $\Ap^R$-equivariant way, the elements of $\Vect(\hom)$ naturally lift to elements of $\Vect(\lGp)$ which commute with $\abelp^R$. Again, let $\eta_a$ denote the lift of $\eta_a$ in $\Vect(\lGp)$. Then by Lemma \ref{4.5}, the lift $\eta_a$ lies in $\liep$ and, moreover, $[u^L,\eta_a]=\sum_{\al I\sqcup \Delta_0} F_\al(\eta_a)[u,e_\al]^L$ with $F_\al(\eta_a)\in\C[\hom]$ for $u\in\glau_1$. Since $\bar{s}\in \abelp$, we have, for $u\in\glau_1$,
\begin{align*}
[u^L, [\bar{s}^R, \eta_a]]&=[\bar{s}^R, [u^L, \eta_a]]=\bigl[\bar{s}^R, \sum_\al F_\al(\eta_a) [u, e_\al]^L\bigr]
                          = \sum_\al \bar{s}^R(F_\al(\eta_a)) [u, e_\al]^L\\
                          &=0.
\end{align*}
Since $\glau_1$ generates $\lgp$, we obtain $[\bar{s}^R, \eta_a]\in \Vect(\lGp)^{\lgp^L}=\lgp^R$.
The degree of $[\bar{s}^R, \eta_a]$ is non-positive and the degrees of the elements in $\lgp^R$ are positive with respect to the action of $x+(d+1) L_0$. This implies $[\bar{s}^R, \eta_a]=0$, i.e., $\eta_a\in \liep^{\bar{s}^R}$. 
On the other hand, we have $\liep^{\bar{s}^R}=\witt_-^{\bar{s}^R}\ltimes (\glau^{\bar{s}})^R$ since $[\bar{s}^R, \witt_-]\subset[\abel, \witt_-]^R \subset \abel^R$ and $[\bar{s}^R,\glau^R]\cap\abel^R=0$.
Since $\bar{s}\in \glau$ is not an element in $\g$, we have $\witt_-^{\bar{s}^R}=0$. From the definition of $\bar{s}$ (and $s$), we obtain $\glau^{\bar{s}^R}= \abel^R$. Thus $\liep^{\bar{s}^R}=\abel^R$ holds. Therefore,  $\eta_a$ lies in $\abel^R$, which means $\int H(\abeln)\subset \abel^R$. Considering the degree given by the action of $x+(d+1)L_0$, we see $\int H(\abeln)\subset \abeln^R$. By Lemma \ref{1.4}, it is easy to show that the map $\int$ on $ H(\abeln)$ is injective. Since the map $a\mapsto \int H(a)$ preserves the degree with respect to the action of $x+(d+1) L_0$, we obtain $\int H(\abeln)=\abeln^R$.
\end{proof}

\end{document}